\documentclass[11pt,reqno]{amsart}
\title[Finitary isomorphisms of renewal point processes and regenerative processes]
{Finitary isomorphisms of renewal point processes and continuous-time regenerative processes}
\date{\today}

\author{Yinon Spinka}
\address{University of British Columbia.
    Department of Mathematics.
 	Vancouver, BC V6T 1Z2, Canada.}
    \email{yinon@math.ubc.ca}
    \thanks{This work was supported in part by NSERC of Canada.}

\usepackage{amsmath, amsthm, amsopn, amssymb, microtype}
\usepackage{enumitem, tikz, etoolbox, intcalc, geometry, caption, subcaption, afterpage}
\usepackage[english]{babel}

\usepackage[colorlinks=true,urlcolor=blue,pdfborder={0 0 0}]{hyperref}
\hypersetup{linkcolor=[rgb]{0,0,0.6}}
\hypersetup{citecolor=[rgb]{0,0.6,0}}

\usepackage[nameinlink]{cleveref}
  \crefname{theorem}{Theorem}{Theorems}
  \crefname{thm}{Theorem}{Theorems}
  \crefname{mainthm}{Theorem}{Theorems}
  \crefname{lemma}{Lemma}{Lemmas}
  \crefname{lem}{Lemma}{Lemmas}
  \crefname{remark}{Remark}{Remarks}
  \crefname{prop}{Proposition}{Propositions}
  \crefname{defn}{Definition}{Definitions}
  \crefname{corollary}{Corollary}{Corollaries}
  \crefname{cor}{Corollary}{Corollaries}
  \crefname{section}{Section}{Sections}
  \crefname{figure}{Figure}{Figures}

\newtheorem{thm}{Theorem}
\newtheorem{lemma}[thm]{Lemma}
\newtheorem{prop}[thm]{Proposition}

\newtheorem{cor}[thm]{Corollary}

\newtheorem{quest}[thm]{Question}

\theoremstyle{definition}
\newtheorem{remark}{Remark}

\newcommand{\cL}{\mathcal{L}}
\newcommand{\cF}{\mathcal{F}}

\newcommand{\N}{\mathbb{N}}
\newcommand{\M}{\mathbb{M}}
\newcommand{\R}{\mathbb{R}}
\newcommand{\Z}{\mathbb{Z}}

\newcommand{\E}{\mathbb{E}}

\renewcommand{\Pr}{\mathbb{P}}
\newcommand{\1}{\mathbf{1}}
\def\eqd{\,{\buildrel d \over =}\,}

\newcommand{\fiso}{\simeq_{\text{f}}}

\newcommand{\iid}{i.i.d.}

\begin{document}

\begin{abstract}
	We show that a large class of stationary continuous-time regenerative processes are finitarily isomorphic to one another. The key is showing that any stationary renewal point process whose jump distribution is absolutely continuous with exponential tails is finitarily isomorphic to a Poisson point process. We further give simple necessary and sufficient conditions for a renewal point process to be finitarily isomorphic to a Poisson point process.
This improves results and answers several questions of Soo~\cite{soo2019finitary} and of Kosloff and Soo~\cite{kosloff2019finitary}.
\end{abstract}

\maketitle

\section{Introduction and main results}\label{sec:introduction}

An important problem in ergodic theory is to understand which measure-preserving systems are isomorphic to which other measure-preserving systems. The measure-preserving systems considered here are the flows associated to continuous-time processes and point processes, equipped with the action of $\R$ by translations. The simplest and most canonical of these processes are the Poisson point processes.\footnote{All Poisson point processes in this paper are homogeneous, i.e., have constant intensity.} It is therefore of particular interest to identify those processes which are isomorphic to a Poisson point process.
A consequence of Ornstein theory~\cite{ornstein2013newton} is that any two Poisson point processes are isomorphic to one another. In fact, Poisson point processes are examples of infinite-entropy Bernoulli flows, and any two infinite-entropy Bernoulli flows are isomorphic.


The factor maps provided by Ornstein theory are obtained through abstract existence results. Consequently, such maps are non-explicit and one usually has no control on the coding window, i.e., the size of the window one needs to observe in the source process in order to be able to determine the target process on a fixed interval. Indeed, there is no a priori reason to expect that the coding window be finite. A factor map is said to be \textbf{finitary} if the coding window is almost surely finite, and two processes are \textbf{finitarily isomorphic} if there is an isomorphism from one to the other such that both it and its inverse are finitary. The notion of finitary isomorphism is strictly stronger than that of isomorphism -- two processes may be isomorphic, yet not finitarily isomorphic. As mentioned above, Poisson point processes are of fundamental importance, and hence, so is the following result by Soo and Wilkens.

\begin{thm}[Soo--Wilkens~\cite{soo2018finitary}]\label{thm:PPP}
Any two Poisson point processes are finitarily isomorphic.\footnote{The result in~\cite{soo2018finitary} is for $d$-dimensional Poisson point processes endowed with the group of isometries of $\R^d$; we only consider one-dimensional processes endowed with the group of translations.}
\end{thm}

A basic class of continuous-time processes are continuous-time Markov chains. A continuous-time Markov chain on a finite state-space can be characterized by an initial distribution, a holding rate for each state and a transition matrix with all diagonal entries being zero. If at some time the process is in some state, then after an exponentially distributed holding time it jumps to a new state according to the transition probabilities (with the holding time and the choice of the new state independent of each other and of the history of the process). The embedded chain is the discrete-time Markov chain obtained by observing the continuous-time process at the jump times. A continuous-time Markov chain is irreducible if its embedded chain is. Feldman and Smorodinsky~\cite{feldman1971bernoulli} proved that stationary irreducible continuous-time Markov chains on finite state-spaces are infinite-entropy Bernoulli flows, and hence isomorphic to Poisson point processes.
Soo proved a finitary version of this result, under the additional assumptions that the holding rates are identical for all states and that the embedded Markov chain is aperiodic.

\begin{thm}[Soo~\cite{soo2019finitary}]\label{thm:Soo-Markov}
Any stationary irreducible continuous-time Markov chain with finitely many states, uniform holding rates and an aperiodic embedded chain is finitarily isomorphic to any Poisson point process.
\end{thm}

Soo asked about the possibility of dropping the latter two assumptions. As we shall see, a consequence of our main results is that this is indeed possible, thereby yielding the finitary counterpart to Feldman and Smorodinsky's result (see \cref{cor:MC}).

Towards motivating our approach to establishing such results, we first note that there is an intimate relation between Poisson point processes and continuous-time Markov chains with uniform holding rates -- the jump times of the latter constitute a Poisson point process, which, moreover, is independent of the embedded Markov chain.
These two properties are at the heart of Soo's proof of \cref{thm:Soo-Markov}. In general, when the holding rates vary between different states, both properties are lost. Instead, to recover some form of independence, we aim to find a subset of the jump times, where the process starts anew. To this end, we consider the underlying point process consisting of the times at which the process jumps to a given state. Since the process is a continuous-time Markov chain, any point process obtained in this way is a \textbf{renewal point process}; this is a simple point process in which the distances between consecutive points are independent and identically distributed. We call this common distribution the \textbf{jump distribution}.
Thus, while a continuous-time Markov chain with uniform holding rates is intimately-related to a Poisson point process, a continuous-time Markov chain with arbitrary holding rates is similarly related to some renewal point process. Furthermore, the jump distribution of the latter renewal point process is easily seen to be absolutely continuous with exponential tails, where a jump distribution $T$ is said to have \textbf{exponential tails} if $\Pr(T \ge t) \le Ce^{-ct}$ for some constants $C,c>0$ and all $t \ge 0$.
For this reason, the following result about renewal point processes is key.

\begin{thm}\label{thm:RP}
Any stationary renewal point process whose jump distribution is absolutely continuous with exponential tails is finitarily isomorphic to any Poisson point process.
\end{thm}

While the assumption of exponential tails is necessary, the absolute continuity is not. Our next result provides a simple necessary and sufficient condition for a renewal point process to be finitarily isomorphic to a Poisson point process.
A jump distribution $T$ is \textbf{non-singular} if it has a non-zero absolutely continuous component.
For a positive integer $k$, the \textbf{$k$-th convolution power} of $T$ is the distribution of a sum of $k$ independent copies of $T$. We say that $T$ has a \textbf{non-singular convolution power} if its $k$-th convolution power is non-singular for some $k$.
We remark that there exist singular jump distributions with non-singular convolution powers, and there exist singular continuous jump distributions all of whose convolutions powers are singular.



\begin{thm}\label{thm:RP2}
A stationary renewal point process is finitarily isomorphic to a Poisson point process if and only if its jump distribution has exponential tails and a non-singular convolution power.
\end{thm}

\cref{thm:RP2} gives an answer to Question~4 in~\cite{kosloff2019finitary}, which asked to find general conditions for a renewal point process to be finitarily isomorphic to a Poisson point process.
We mention that the necessary conditions for being finitarily isomorphic to a Poisson point process are also necessary for being a finitary factor of a Poisson point process (see \cref{prop:necessary}).
We also mention that a discrete-time version of the finitary factor problem was considered by Angel and the author~\cite{angel2019markov}.

\medbreak

\cref{thm:RP} tells us that the times at which a continuous-time Markov chain jumps to a given state can be finitarily coded by a Poisson point process. However, it does not provide a means by which to code the excursions of the continuous-time Markov chain in between such times. To encode this additional information, we consider point processes with additional structure.

A \textbf{marked point process} is a point process in which each point of the process has an associated value, which we call a mark. The marks are assumed to be elements in a standard Borel space.
A marked point process is \textbf{simply-marked} if, given the points of the process, the marks of the points are conditionally independent, and the conditional distribution of the mark of any point depends only on the distance to the nearest point to its left.
Our next result is an extension of (the if part of) \cref{thm:RP2} to simply-marked renewal point processes.

\begin{thm}\label{thm:simply-marked-RP}
Any stationary simply-marked renewal point process whose jump distribution has exponential tails and a non-singular convolution power is finitarily isomorphic to any Poisson point process.
\end{thm}

\cref{thm:simply-marked-RP} will easily follow from \cref{thm:RP2} and a technique for marking point processes, which borrows from ideas developed by Ball~\cite{ball2005poisson}, Kosloff--Soo~\cite{kosloff2019finitary} and Soo~\cite{soo2019finitary}. This simple extension of (the if part of) \cref{thm:RP2} turns out to be very useful. Indeed, given a continuous-time Markov chain, by viewing each mark as the excursion of the process between the previous and current point, a continuous-time Markov chain can be seen as a simply-marked renewal point process. In a similar manner, many processes can be naturally viewed as simply-marked renewal point processes to which \cref{thm:simply-marked-RP} can then be applied.
We now describe a class of such processes called regenerative processes. Intuitively, a regenerative process is a process which can be decomposed into \iid\ segments. Several variants of the definition exist in the literature, and we give another definition which is relevant for our purposes. To avoid technicalities, we restrict our attention to stationary processes $X=(X_t)_{t \in \R}$ taking values in a Polish space and having paths that are almost surely right-continuous with left-hand limits.

We say that $X$ is \textbf{regenerative} if there exists a renewal point process $\Lambda$ such that, given the points $t_1<t_2<\cdots$ of $\Lambda$ in $(0,\infty)$, the segments $\{(X_t)_{t_n \le t < t_{n+1}}\}_n$ between consecutive points are conditionally independent, and the conditional distribution of each segment $(X_{t+t_n})_{0 \le t < t_{n+1}-t_n}$ depends only on its length $t_{n+1}-t_n$.
Such a $\Lambda$ is said to be regenerative for $X$.
While some definitions allow $\Lambda$ to have additional randomness, we always require that $\Lambda$ is measurable with respect to $X$. In fact, we will only be interested in a more restrictive notion.
We say that $X$ is \textbf{finitarily-regenerative} if there exists such a $\Lambda$, which, in addition, is a finitary factor of $X$.
In this case, we say that $\Lambda$ is finitarily-regenerative for $X$, and we say that the jump distribution of $\Lambda$ is an associated jump distribution of $X$.
Note that $X$ may have more than one associated jump distribution.


\begin{thm}\label{thm:regenerative}
Any stationary finitarily-regenerative process having an associated jump distribution that has exponential tails and a non-singular convolution power is finitarily isomorphic to any Poisson point process.
\end{thm}

We mention that while we have formulated \cref{thm:regenerative} for continuous-time processes, it also applies to marked point processes, where a finitarily-regenerative marked point process is defined similarly. Alternatively, one may view a marked point process as a continuous-time process to which \cref{thm:regenerative} may then be applied directly when it satisfies the required assumptions.

\medskip

As with many previous constructions of finitary factors and isomorphisms (e.g., \cite{ball2005poisson,kosloff2019finitary,soo2019finitary,angel2019markov}), our constructions too are elementary and mostly explicit. While our proof borrows some ideas from these constructions, it does not rely on previous results about finitary factors and isomorphisms, and therefore provides a self-contained proof.

\medskip

The rest of the paper is organized as follows. In \cref{sec:app}, we introduce some further results, all of which are applications of the above theorems. In \cref{sec:preliminary}, we provide the required preliminary definitions and some background on finitary isomorphisms. In \cref{sec:simply-marked}, we obtain results relating renewal point processes to their simply-marked versions, and then prove \cref{thm:simply-marked-RP} and \cref{thm:regenerative} using these and \cref{thm:RP2}. The proof of the `if' part of \cref{thm:RP2} is split into the two parts: the first part, given in \cref{sec:renewal-with-regular-jd}, shows that \cref{thm:RP} holds under an additional regularity assumption on the jump distribution; the second part, given in \cref{sec:regularize-jd}, shows how one may regularize the jump distribution in order to reduce the theorem to the former situation. In \cref{sec:proof-main-thm}, we put the two parts together to complete the proof of the `if' part of \cref{thm:RP2} and also give the proof of the `only if' part.

\section{Applications}
\label{sec:app}

In this section, we describe some applications of our main theorems. These are mostly immediate consequences of the main theorems and the main goal here is to give a flavor of the possibilities by pointing out various classes of processes to which our results apply.
When a marked point process has its marks in a finite space, we sometimes refer to the marks as \textbf{colors} and to the marked point process as a \textbf{colored point process}.


We begin with an application to a simple class of colored point processes called alternating point processes.
The \textbf{alternating point process} corresponding to a stationary point process is obtained by coloring the points of the latter red and blue, alternating between the two colors, and with any given point receiving each color with equal probability. In particular, given the uncolored points, the sequence of colors is one of two equally-likely possibilities.
It is therefore reasonable to expect that one cannot obtain an alternating Poisson point process as a factor of a Poisson point process in which the (uncolored) points of the former coincide with those of the latter.
Indeed, Holroyd, Pemantle, Peres and Schramm~\cite[Lemma~11]{holroyd2009poisson} prove that this is not possible.
On the other hand, if one does not insist that the colored points are the same points as the original Poisson point process, then Ornstein theory implies that this is possible. Indeed, it follows from a result of Feldman and Smorodinsky~\cite{feldman1971bernoulli} that alternating Poisson point processes are Bernoulli.
The question of whether such processes are finitarily isomorphic to Poisson point processes was raised in~\cite{soo2019finitary}.
While an alternating Poisson point process is not a simply-marked point process (so that \cref{thm:simply-marked-RP} does not apply), it is a finitarily-regenerative process. Indeed, the set of red-colored points is (finitarily-)regenerative for the process. Since the red-colored points constitute a stationary renewal point process whose jump distribution is the sum of two independent exponential random variables, \cref{thm:regenerative} implies that alternating Poisson point processes are in fact finitarily isomorphic to Poisson point processes, thereby affirmatively answering Question~1 of~\cite{soo2019finitary}.

\begin{cor}\label{cor:alternating}
Any Poisson point process is finitarily isomorphic to its alternating point process. More generally, any alternating renewal point process whose jump distribution has exponential tails and a non-singular convolution power is finitarily isomorphic to any Poisson point process.
\end{cor}

An alternating Poisson point process can be seen as a continuous-time Markov chain. Indeed, in this case, the embedded chain deterministically alternates between the red and blue states, with transitions occurring at a constant rate. It is straightforward that a stationary irreducible continuous-time Markov chain on a finite state-space is a regenerative process, with the return times to any fixed state being (finitarily-)regenerative and constituting a stationary renewal point process having an absolutely continuous jump distribution with exponential tails. Thus, \cref{thm:regenerative} yields that such processes are finitarily isomorphic to Poisson point processes. This leads to the following generalization of \cref{thm:Soo-Markov} without the assumptions that the holding rates are uniform and that the embedded chain is aperiodic.

\begin{cor}\label{cor:MC}
Any stationary irreducible continuous-time Markov chain with finitely many states is finitarily isomorphic any Poisson point process.
\end{cor}

Similarly, an alternating renewal point process is a special case of a stationary colored renewal point process in which the $\Z$-process induced by the colors, the so-called \textbf{skeleton}, is a (discrete-time) Markov chain that is independent of the points of the process. In the case when the skeleton is a general irreducible Markov chain on a finite state-space, as before, the points of any fixed state/color are regenerative and constitute a stationary renewal point process. This setting includes continuous-time Markov chains with uniform holding rates, but does not include the case where the holding rates depend on the states. To accommodate the latter, we consider colored point process which we call Markov-colored point processes.

A \textbf{Markov-colored point process} is a colored point process in which the color of and distance to the next point depends on the history of the process only through the color of the current point.
Such a process may also be described as follows: for each color $i$, there is a jump distribution $T_i$ and a random variable $C_i$ taking values in the set of colors, and if the process has a point of color $i$ at some time, then (independently of the history of the process up to this time) the next point appears after a holding time $T_i$ and receives color $C_i$.

We note that when $T_i$ is exponential, $C_i$ is almost surely not equal to $i$, and $T_i$ and $C_i$ are independent of each other, this gives an equivalent description of a continuous-time Markov chain. In general, we need not make any of these assumptions: $T_i$ may be any jump distribution, $C_i$ may equal $i$ with positive probability, and $(T_i,C_i)$ may be coupled in any manner. Either way, the skeleton of a Markov-colored point process is always a discrete-time Markov chain.
Thus, Markov-colored point processes generalize continuous-time Markov chains, allowing non-exponentially-distributed holding times, allowing a state to be reentered without leaving it, and allowing dependence between the random state to which the process jumps and the holding time leading to that jump.
We say that a Markov-colored point process is irreducible if its skeleton is.


As with a continuous-time Markov chain, in this setting too, given a stationary irreducible Markov-colored point process $X$, the points of any fixed color are (finitarily-)regenerative and constitute a stationary renewal point process. We say that the jump distribution of the latter is an associated jump distribution of $X$.
When this jump distribution has exponential tails and a non-singular convolution power, \cref{thm:regenerative} applies and shows that $X$ is finitarily isomorphic to a Poisson point process. Thus, we have the following generalization of \cref{cor:MC}.

\begin{cor}\label{cor:MCPP}
Any stationary irreducible Markov-colored point process having an associated jump distribution that has exponential tails and a non-singular convolution power is finitarily isomorphic to any Poisson point process.
\end{cor}

\smallskip

Let us now consider continuous-time Markov chains with countably many states (all Markov chains here are implicitly assumed to be right-continuous). The assumption of finitely many states in \cref{cor:MC} is not essential and was only used to ensure \textbf{exponential recurrence} of the continuous-time Markov chain. By this we mean that the chain is recurrent and that starting from some state, the time it takes to escape and then return to that state has exponential tails. For instance, if the holding rates are uniformly bounded below and the embedded discrete-time Markov chain has exponential return times, then the continuous-time Markov chain is exponentially recurrent. In fact, as long as the chain is recurrent (even if the embedded chain is null-recurrent), one may always take the holding rates to grow sufficiently fast so that the continuous-time Markov chain is exponentially recurrent.
As with chains on finite state-spaces, given any stationary irreducible recurrent continuous-time Markov chain on a countable state-space, the return times to any fixed state are (finitarily-)regenerative for the chain and constitute a stationary renewal point process having an absolutely continuous jump distribution.
\cref{thm:regenerative} thus yields the following.

\begin{cor}
Any stationary irreducible exponentially-recurrent continuous-time Markov chain with countably many states is finitarily isomorphic to any Poisson point process.
\end{cor}

We mention that even if the embedded chain is transient, one may still define a (non-minimal) stationary continuous-time Markov chain in such a way that it satisfies the assumptions of \cref{thm:regenerative}. One possible way to do this is to choose the holding rates large enough so that the total time of an excursion from a given state $s$ has exponential tails; such an excursion ends either by returning to $s$ after finitely many transitions, or by an explosion occurring after infinitely many transitions (which do not lead back to $s$) in a finite time period. After every such excursion, another independent excursion from $s$ is immediately started. In this case, the times at which the process enters state $s$ constitute a regenerative renewal point process and \cref{thm:regenerative} is applicable (by considering the one-point compactification of the state space, the sample paths are guaranteed to have left-hand limits almost surely).
In fact, since we do not need the process to be Markov, we could even make the holding rates depend on the step within the excursion. For example, if $(\lambda_n)_n$ is a sequence of positive numbers such that $\sum \frac{1}{\lambda_n} < \infty$, and we let the $n$-th transition within an excursion (if it exists) have holding rate $\lambda_n$, then the total time of a single excursion is at most $\sum E_n$ (regardless of the embedded chain), where the $E_n$ are independent exponential random variables with rate $\lambda_n$. In this case, the return time to $s$ is absolutely continuous with exponential tails, so that \cref{thm:regenerative} shows that such processes are also finitarily isomorphic to Poisson point processes.

\begin{remark}
It was remarked in~\cite{soo2019finitary} that \cref{thm:Soo-Markov} extends to the case of countably many states with the additional assumption that the embedded chain has exponential return times. Indeed, the proof there is based on applying a finitary isomorphism to the embedded chain, appealing to Keane--Smorodinsky's~\cite{keane1979finitary} (finite state-space) or Rudolph's~\cite{rudolph1982mixing} (countable state-space) result for discrete-time Markov chains. On the other hand, our method of proof does not separate the embedded chain from the continuous-time chain, but rather works directly with the continuous-time process. This approach has two advantages: First, it does not rely on previous results for $\Z$-processes and therefore yields a self-contained proof. Second, it allows to get rid of the assumptions that the embedded chain is aperiodic and that the holding rates are uniform.
\end{remark}

Let us now discuss applications to Markov processes on continuous state spaces.
Consider a stationary Markov process $X=(X_t)_{t \in \R}$ taking values in a Polish space $S$ and having paths that are right-continuous with left-hand limits.
A natural way to find a regenerative set $\Lambda$ for $X$ is via stopping times. If $T$ is a stopping time, then for any initial deterministic time $T_0 \in \R$, we can construct a sequence $T_0<T_1<T_2<\cdots$ by $T_{n+1} = T_n + T \circ \theta_{T_n}$, where $\theta_t$ is the translation by $t$ (and $X$ is defined on its canonical space). If the point process $\{T_1,T_2,\dots\}$ stabilizes on any bounded interval as $T_0 \to -\infty$, then we obtain a stationary point process $\Lambda$ in the limit. For \cref{thm:regenerative} to be applicable in this situation, we would still need to know that $\Lambda$ is a finitary factor of $X$, and that it is a renewal point process that is regenerative for $X$. The former will usually be evident from the definition, while the latter may rely on a strong Markov property.

Let us demonstrate this with two examples of interest: reflected and periodic Brownian motions.
One may obtain reflected Brownian motion on $[0,h]$ for some $h \in (0,\infty)$ by taking a standard Brownian motion $B=(B_t)_{t \in \R}$ and defining $X=(X_t)_{t \in \R}$ by letting $X_t$ be the distance from $B_t$ to the nearest integer multiple of $2h$. By choosing the distribution of $B_0$ to be uniform on $[0,h]$, one further ensures that $X$ is stationary. Periodic Brownian motion on $[0,h]$ is obtained similarly by taking $X_t$ to be $B_t$ mod $h$.
Kosloff and Soo~\cite{kosloff2019finitary} proved that two stationary reflected (periodic) Brownian motions on intervals $[0,h_1]$ and $[0,h_2]$ are finitarily isomorphic whenever $h_1/h_2$ is rational, and that a stationary reflected Brownian motion on $[0,h]$ is finitarily isomorphic to a stationary periodic Brownian motion on $[0,2h]$. As we now explain, our results show that any such process (reflected or periodic, and with any $h$) is finitarily isomorphic to any Poisson point process, and hence also to any other such process.


Let $X$ be a stationary reflected Brownian motion on $[0,h]$.
Let $T$ be the first hitting time of $0$ after hitting $h$, namely, $T=\inf \{ t>0 : X_t = 0\text{ and }X_{t'}=h\text{ for some }t' \in (0,t)\}$. It is clear that $T$ is a stopping time and that the point process $\Lambda$ obtained from $T$ as above is well-defined. In fact, $\Lambda$ is a finitary factor of $X$, since the restriction of $\Lambda$ to $[a,b]$ is determined by $(X_t)_{R_a \le t \le b}$, where $R_a = \sup\{t<a : X_t=h\}$. Moreover, it is standard that $X$ is a Feller process, and hence enjoys the strong Markov property, thereby implying that $\Lambda$ is regenerative for $X$.
Thus, $\Lambda$ decomposes $X$ into independent excursions from 0 to $h$ and back.
Finally, it is not hard to check that the jump distribution of $\Lambda$ is absolutely continuous with exponential tails. Thus, \cref{thm:regenerative} yields that $X$ is finitarily isomorphic to a Poisson point process. A similar reasoning applies to periodic Brownian motion. This yields the following corollary, which answers Questions~1 and~2 of~\cite{kosloff2019finitary} affirmatively.

\begin{cor}\label{cor:Brownian}
Any stationary reflected or periodic Brownian motion on a bounded interval is finitarily isomorphic to any Poisson point process.
\end{cor}

The construction described above, in which reflected Brownian motion is decomposed into independent excursions from 0 to $h$ and back, appeared in~\cite{kosloff2019finitary}, where it was used to obtain the aforementioned results regarding finitary isomorphisms between reflected/periodic Brownian motions.
The renewal point process $\Lambda$, defined above, records the times at which an excursion from 0 to $h$ and back begins. The closely-related renewal point process $\Lambda'$, in which one additionally records the times at which each such excursion first reaches $h$, was called a Brownian excursion point process (of parameter~$h$) in~\cite{kosloff2019finitary}, where its alternating version played an important role. \cref{thm:RP} and \cref{cor:alternating} imply that any (alternating or unmarked) Brownian excursion point process is finitarily isomorphic to a Poisson point process. This gives an affirmative answer to Question~3 of~\cite{kosloff2019finitary}.

We mention that different stopping times could be used instead of the above $T$ in order to obtain \cref{cor:Brownian}. One such example is the first hitting time of $0$ after an initial delay of time $t_0>0$, namely, $T' = \inf \{t \ge t_0 : X_t=0\}$. Indeed, the renewal point process it yields is regenerative for the same reason as before, and is a finitary factor of $X$ since there are intervals of length $t_0$ in which $X$ does not vanish.

It is clear that the above type of reasoning may be applied to various processes. We give one last example of such a process whose paths are not continuous and which is not in a fixed state at the regeneration times. Let $\Lambda$ be a stationary renewal point process whose jump distribution $T$ has exponential tails and a non-singular convolution power. For each interval $[a,b)$ of consecutive points of $\Lambda$, let $(X_t)_{a \le t < b}$ be an independent Brownian motion (independent of those of other intervals and independent of $\Lambda$) with $X_a$ chosen according to some fixed distribution~$I$. Since $\Lambda$ coincides with the set of discontinuity points of $X$, and since it is regenerative by construction, \cref{thm:regenerative} shows that $X$ is finitarily isomorphic to a Poisson point process.
Note that, even when $T$ is exponential (in which case $X$ is a strong Markov process), not all stopping times give rise to regenerative sets for $X$; for example, the set of discontinuity points in which the jump discontinuity has size at most 1 is not regenerative when $I$ is non-constant.
On the other hand, it was not important that the segments $(X_t)_{a \le t < b}$ were Brownian motion, and we could have chosen $(X_{a+t})_{0 \le t < b-a}$ according to any other continuous process in which the distribution at any positive time is continuous.

\begin{remark}
Going through the proof of \cref{thm:RP2}, it is straightforward to check that the finitary isomorphism guaranteed by the theorem is very efficient in that its coding window has exponential tails (as does the coding window of its inverse). In particular, this also means that such a finitary isomorphism exists between any two Poisson point processes.
The same is true for \cref{thm:simply-marked-RP}, as well as for \cref{thm:regenerative} if one assumes that the underlying renewal point process $\Lambda$ is not just a finitary factor of $X$, but that its coding window also has exponential tails. All applications of \cref{thm:regenerative} mentioned in this section satisfy this.
\end{remark}

\section{Preliminary definitions and background}
\label{sec:preliminary}

\subsection{Point processes}

Let $\M$ denote the space of locally-finite simple-point Borel measures on~$\R$. Every $\mu \in \M$ is a countable (or finite) sum of Dirac point masses with no multiplicity, and can thus be identified with a countable (or finite) set of isolated points $[\mu]=\{p_i\}_i \subset \R$.
Let $\Theta=(\theta_t)_{t \in \R}$ denote the group of translations of $\R$, where $\theta_t \colon \R \to \R$ is given by $\theta_t(s)=s-t$. This group acts on $\M$ by $\theta_t(\mu) = \mu \circ \theta_{-t}$, or equivalently, by $\theta_t(\{p_i\}_i) = \{p_i-t\}_i$.
A (simple) \textbf{point process} on $\R$ is a random variable $\Lambda$ that takes values in $\M$. The point process $\Lambda$ is \textbf{stationary} if its distribution is preserved under the action of $\Theta$.
A stationary point process $\Lambda$ can then be seen as a measure-preserving system $(\M,\Pr(\Lambda \in \cdot),\Theta)$.

Let $\Lambda$ be a point process and suppose that its points $[\Lambda]=\{p_i\}_i$ are indexed by an interval in $\Z$ so that $\cdots < p_{-1}< p_0 \le 0 < p_1 < p_2 < \cdots$. When $[\Lambda]$ is almost surely unbounded above, $p_i$ is a well-defined random variable for every $i \ge 1$, with $p_i \to \infty$ as $i \to \infty$ almost surely. Similarly, when $[\Lambda]$ is almost surely unbounded below, $p_i$ is a well-defined for every $i \le 0$, with $p_i \to -\infty$ as $i \to -\infty$ almost surely.
A point process is a \textbf{renewal point process} if $\{p_{i+1}-p_i\}_{i=1}^\infty$ are \iid\ random variables. The common law of these variables is called the \textbf{jump distribution}.
In a stationary renewal point process, $\{p_{i+1}-p_i\}_{i \neq 0}$ are \iid\ random variables having the same distribution as the jump distribution, whereas $p_1-p_0$ has a size-biased distribution. The most notable example of a stationary renewal point process is the \textbf{Poisson point process} of intensity $\lambda>0$, obtained when the jump distribution is the exponential distribution with rate $\lambda$. 

A \textbf{marked point process} is a point process $\Pi$ on $\R \times S$ for some Polish space $S$, i.e., a random variable taking values in the space of Borel locally-finite simple-point measures on $\R \times S$.
As before, $\Pi$ can be identified with a countable set of isolated points $[\Pi]=\{(p_i,m_i)\}_i \subset \R \times S$.
The translation group $\Theta$ acts on marked point processes by acting on the first coordinate, i.e., $\theta_t(\{(p_i,m_i)\}_i) = \{(p_i-t,m_i)\}_i$.
We will be interested only in marked point processes whose projection onto $\R$ is a point process on $\R$ (i.e., the projection does not create overlapping points or accumulation points) and we implicitly assume this in the definition of a marked point process. Such a marked point process can be thought of as a point process on $\R$ in which every point carries a value in $S$ called its \textbf{mark}.
We say that $[\Lambda]=\{p_i\}_i$ is the unmarked point process underlying $\Pi$, and that the point $p_i$ has mark $m_i$.
Suppose now that $[\Lambda]$ is almost surely unbounded above and below, and let $\{p_i\}_i$ be indexed as before. The discrete-time process $(m_i)_{i \in \Z} \in S^\Z$ called the \textbf{skeleton} of $\Pi$.
We say that a marked point process $\Pi$ is \textbf{simply-marked} if, given $\Lambda$, the marks $(m_i)_i$ are conditionally independent, with the distribution of each $m_i$ depending on $\Lambda$ only through $p_i-p_{i-1}$.
If the latter distribution does not at all depend on $\Lambda$, then $(m_i)_i$ are \iid\ and independent of $\Lambda$, and we say that the point process is \textbf{independently-IID-marked}.
We stress that by a marked renewal point process, we mean that the underlying unmarked point process is a renewal point process. Note that a stationary marked point process is a simply-marked renewal point process if and only if $\{(p_{i+1}-p_i,m_{i+1})\}_{i=1}^\infty$ are \iid\ random variables. The common law of these variables is called the \textbf{jump-mark distribution}. Note also that a simply-marked renewal point process is independently-IID-marked if and only if the jump-mark distribution is a product distribution.
We say that a marked point process $\Pi$ is a \textbf{Markov-marked point process} if, for any $i \ge 1$, the conditional distribution of $(p_{i+1},m_{i+1})$ given $\{(p_j,m_j)\}_{j \le i}$ depends only on $m_i$. Note that a stationary simply-marked renewal point process is precisely a Markov-marked point process in which the latter conditional distribution does not depend on $m_i$.

When we use the term \emph{jump distribution} outside the context of a renewal point process, we simply mean a random variable taking values in $(0,\infty)$.

\subsection{Continuous-time processes}

Besides point processes, we also consider continuous-time processes $X=(X_t)_{t \in \R}$ taking values in a Polish space $S$. We always assume these processes to have paths which are almost surely right-continuous with left-hand limits. Thus, $X$ is a random variable taking values in the space $D$ of functions $f \colon \R \to S$ which are right-continuous with left-hand limits. The space $D$ is endowed with the Skorokhod topology which turns it into a Polish space. The group of translations $\Theta$ acts on $D$ by $\theta_t(f) = f \circ \theta_{-t}$. The process $X$ is stationary if its distribution is preserved under the action of $\Theta$. A stationary process $X$ can then be viewed as a measure-preserving system $(D,\Pr(X \in \cdot),\Theta)$.

We note that a marked point process $\Pi$ can be seen as continuous-time process $X$ whose paths are right-continuous with left-hand limits. Suppose that the marks of $\Pi$ take values in a Polish space $S$.
If consecutive points of $\Pi$ always have different marks, then we can simply take $X$ to be the $S$-valued process defined by $X_t = s$, where $s$ is the mark of the largest point of $\Pi$ in $(-\infty,t]$. If consecutive points may have identical marks (which is always the case for an unmarked point process), then we can take $X$ to be the $(S \times [0,\infty))$-valued process defined by $X_t = (s,t-t')$, where $t'$ is the largest point of $\Pi$ in $(-\infty,t]$ and $s$ is its mark. In this manner, results about continuous-time processes can also be applied to marked point processes.

\subsection{Factors and isomorphisms}

We begin by defining the notion of a factor from one measure-preserving system $\mathsf M = (\mathcal M,\mu,M)$ to another $\mathsf N = (\mathcal N,\nu,N)$, where both are given by a single transformation. A \textbf{factor} from $\mathsf M$ to $\mathsf N$ is a measurable map $\varphi \colon \mathcal M \to \mathcal N$ such that $\mu \circ \varphi^{-1} = \nu$ and $\varphi \circ M = N \circ \varphi$ on a set of full $\mu$-measure.
We say that $\mathsf N$ is a factor of $\mathsf N$. If $\varphi$ is injective with a measurable inverse, then $\varphi$ is an \textbf{isomorphism} from $\mathsf M$ to $\mathsf N$ and we say that $X$ and $Y$ are isomorphic.

The definitions naturally extend to measure-preserving flows. For example, if $X$ and $Y$ are two stationary point processes, then we identify them with the measure-preserving flows $(\M,\Pr(X \in \cdot),\Theta)$ and $(\M,\Pr(Y \in \cdot),\Theta)$, and a factor from $X$ to $Y$ is a measurable map $\varphi \colon \M \to \M$ such that $\varphi(X)$ and $Y$ are equal in distribution and, on a set of full measure, we have $\varphi \circ \theta_t = \theta_t \circ \varphi$ for all $t \in \R$.

Given a stationary discrete-time process $X=(X_n)_{n \in \Z}$ taking values in some measurable space $A$, we identify $X$ with the measure-preserving system $(A^\Z, \Pr(X \in \cdot), \sigma)$, where $\sigma$ is the left-shift on $A^\Z$.

\subsection{Finitary factors and isomorphisms}

We begin by defining the notion of a finitary factor from one stationary point process to another.
For $\mu \in \M$ and an interval $I \subset \R$, we denote the restriction of $\mu$ to $I$ by $\mu|_I=\mu(\cdot \cap I)$. Recalling that $\mu$ may be viewed as a countable subset $[\mu]$ of $\R$, we may also view $\mu|_I$ as the set $[\mu] \cap I$.
We say that $\mu$ and $\mu'$ agree on $I$ if $\mu|_I=\mu'_I$.

Let $X$ and $Y$ be two stationary point processes.
Intuitively, a factor $\varphi \colon \M \to \M$ from $X$ to $Y$ is \textbf{finitary} if the restriction of $\varphi(X)$ to a bounded interval is almost surely determined by the restriction of $X$ to some, perhaps much larger, bounded interval. To be more precise, a \textbf{coding window} for $\varphi$ is a measurable function $w \colon \M \to \N \cup \{\infty\}$ such that $\varphi(\mu)$ and $\varphi(\mu')$ agree on $[-1,1]$ whenever $\mu$ and $\mu'$ agree on $[-w(\mu),w(\mu)]$. Then $\varphi$ is finitary if there exists a coding window $w$ such that $w(X)$ is almost surely finite. If $\varphi$ is an isomorphism such that both $\varphi$ and $\varphi^{-1}$ are finitary, then $\varphi$ is a \textbf{finitary isomorphism}. In this case, we say that $X$ and $Y$ are finitarily isomorphic and write $X \fiso Y$. We note that in the literature, a finitary isomorphism is sometimes defined to be an isomorphism which is finitary; we require that its inverse is also finitary.

We are also interested in finitary factors and isomorphisms involving other processes. These are defined similarly, with the corresponding notion of agreeing on an interval, e.g., two marked point processes agree on an interval if the points and their marks coincide on that interval, and two processes taking values in a Polish space agree on an interval if their restrictions to that interval coincide.
The notion of finitary factor is similarly defined for discrete-time processes.

\subsection{Background}

Let us begin the discussion with discrete-time processes.
A \textbf{Bernoulli shift} is the discrete-time measure-preserving system associated to an \iid\ process $(Y_n)_{n \in \Z}$ taking values in a finite or countable set. The entropy of a Bernoulli shift is given by the Shannon entropy of $Y_0$. Ornstein proved that any two Bernoulli shifts of equal entropy are isomorphic~\cite{ornstein1970bernoulli,ornstein1970two}. A measure-preserving system consisting of a single transformation is called \textbf{Bernoulli} if it is isomorphic to a Bernoulli shift. Since the Kolmogorov--Sinai entropy of a measure-preserving system is an isomorphism-invariant~\cite{kolmogorov1959entropy,sinai1959concept}, two Bernoulli systems are isomorphic if and only if they have the same entropy.
Furthermore, Ornstein theory provides various criteria by which one may check whether a given system is Bernoulli. Consequently, many systems are known to be Bernoulli, including stationary finite-state mixing Markov chains~\cite{friedman1970isomorphism} and any factor of a Bernoulli system~\cite{ornstein1970factors}.

The analogous theory has also been developed for flows.
A measure-preserving flow $(\mathcal M,\mu,(M_t)_{t \in \R})$ is called Bernoulli if the discrete-time measure-preserving system $(\mathcal M,\mu,(M_{nt})_{n \in \Z})$ is Bernoulli for any $t>0$. The entropy of a measure-preserving flow is the Kolmogorov--Sinai entropy of its time-one map $(\mathcal M,\mu,(M_n)_{n \in \Z})$.
Ornstein proved that any two Bernoulli flows of equal entropy are isomorphic~\cite{ornstein1973isomorphism,ornstein2013newton}. An example of a finite-entropy Bernoulli flow is a stationary renewal point process whose jump distribution takes values 1 and $\sqrt{2}$ with equal probabilities (this is a factor of the corresponding alternating point process, which is equivalent to the Totoki flow~\cite{ornstein1970imbedding}). A canonical example of an infinite-entropy Bernoulli flow is a Poisson point process. Other infinite-entropy Bernoulli flows include stationary irreducible continuous-time Markov chains on finite state-spaces~\cite{feldman1971bernoulli}, and more generally, stationary mixing Markov shifts of kernel type~\cite{ornstein1973mixing}, which include reflected Brownian motion on a bounded interval.

Let us now discuss the finitary counterpart of the theory. We once again begin with discrete-time processes.
In a series of landmark papers~\cite{keane1977class,keane1979bernoulli,keane1979finitary}, Keane and Smorodinsky proved the finitary counterpart of Ornstein's isomorphism theorem for Bernoulli shifts. Namely, they showed that any two \iid\ processes of equal entropy are finitarily isomorphic~\cite{keane1979bernoulli}, and more generally, that any two stationary finite-state mixing Markov chains are finitarily isomorphic~\cite{keane1979finitary}.
This has provided hope of finding finitary counterparts for other parts of Ornstein theory. Indeed, many such counterparts have been found. For example, Rudolph gave an abstract criteria for being finitarily isomorphic to an \iid\ process~\cite{rudolph1981characterization} and applied it to show that stationary countable-state exponentially-recurrent mixing Markov chains are finitarily isomorphic to an \iid\ process~\cite{rudolph1982mixing}, and Shea proved the existence of finitary isomorphisms between some renewal processes and \iid\ processes~\cite{shea2009finitary}.
We stress that there are plenty of other finitary results, of which we have only mentioned a few, and we refer the reader to the survey~\cite{serafin2006finitary} for more background.
Surprisingly, however, it has recently been shown that various parts of the theory, though initially expected to have finitary counterparts, in fact fail to~\cite{gabor2019failure}. This accentuates the fact that it is not always a simple matter to decide whether a finitary counterpart exists, and thereby also adds some excitement to the matter.

For continuous-time processes, finitary results are rather scarce. The most notable results are \cref{thm:PPP} about finitary isomorphisms between two Poisson point processes (building on a previous result of Holroyd--Lyons--Soo~\cite{holroyd2011poisson} and extending a result of Kalikow--Weiss~\cite{kalikow1992explicit}) and \cref{thm:Soo-Markov} about finitary isomorphisms between Poisson point processes and a class of continuous-time Markov chains. Further results include finitary isomorphisms of reflected/periodic Brownian motions on bounded intervals~\cite{kosloff2019finitary} and monotone finitary factors between Poisson point processes~\cite{ball2005poisson}. Our results add to the short list of finitary results for continuous-time processes.

Let us mention that for discrete-time processes taking values in a finite or countable set $A$, having a finite coding window at some point $x \in A^\Z$ is the same as the continuity of the factor map at~$x$. Thus, for such processes, our notion of a finitary factor is the same as an almost everywhere continuous map.
For other processes (either discrete-time processes taking values in an uncountable set, continuous-time processes or point processes), the two notions do not coincide, and we take as the definition the existence of an almost everywhere finite coding window. Note that this is a rather strong requirement (and is therefore sometimes termed \emph{strongly finitary}~\cite{holroyd2011poisson}) as it means that the finitary factor allows to determine the precise output on any bounded interval (i.e., the exact location of the points in a point process, or the exact sample path in a continuous-time process) by observing the input on a large, but bounded, window.

Let us also mention the possibility of obtaining finitary factors which commute with all isometries of the line, i.e., not only translations, but also reflections. For example, the finitary isomorphism guaranteed by \cref{thm:PPP} has this property. Our constructions produce finitary isomorphisms which are translation-equivariant, but not reflection-equivariant. It would be interesting to determine when this is possible.

\begin{quest}
Can the finitary isomorphism in \cref{thm:RP} be made reflection-equivariant?
\end{quest}

\section{Renewal point processes and their simply-marked versions}
\label{sec:simply-marked}

In this section, we provide some results relating certain renewal point processes to their simply-marked versions and to associated finitarily-regenerative processes.
Recall that an independently-IID-marked point process is obtained from an unmarked point process by independently marking each point according to some common law on marks.
The following lemma will allow us to move freely between unmarked and independently-IID-marked renewal point processes, at least when the jump distribution is non-singular.


\begin{lemma}\label{lem:marking}
Let $X$ be a stationary renewal point process with a non-singular jump distribution $T$. Then $X$ is finitarily isomorphic to any of its independently-IID-marked versions.
\end{lemma}

Before proving \cref{lem:marking}, we require the following simple lemma.

\begin{lemma}\label{lem:continuous-cond-distrib}
Let $T_1$ and $T_2$ be independent absolutely continuous jump distributions. Then, given $T_1+T_2$, the conditional distribution of $T_1$ is almost surely absolutely continuous.
\end{lemma}
\begin{proof}
Let $g$ and $h$ be the densities of $T_1$ and $T_2$, respectively. Then $T_1+T_2$ is absolutely continuous with density $f=g*h$ given by convolution of $g$ and $h$.
Thus, given $T_1+T_2=t$, it is straightforward that $T_1$ has conditional density $g_t$ given by $g_t(s)=g(s)h(t-s)/f(t)$.
\end{proof}


\begin{proof}[Proof of \cref{lem:marking}]
Since $T$ is non-singular, there exists a measurable set $A \subset \R$ such that $0<\Pr(T\in A)<1$ and such that $\Pr(T \in \cdot \mid T \in A)$ is absolutely continuous.
Let $x$ be a point of $X$, and let $x_3<x_2<x_1<x$ denote the three nearest points to its left.
Call the point $x$ special if $x-x_1 \in A$, $x_1-x_2 \in A$ and $x_2-x_3 \notin A$. Clearly, if $x$ is special, then $x_1$ and $x_2$ cannot be special (while $x_3$ can be). Note that special points exist almost surely, and that the special point process is a finitary factor of $X$. Furthermore, since $X$ is a renewal point process, the special points are regenerative for $X$. That is, given the set of special points, the restrictions of $X$ to the intervals between consecutive special points are conditionally independent, with the conditional distribution on any such interval depending only on the length of the interval (by an interval between two consecutive points $a$ and $b$, we mean the half-open half-closed interval $(a,b]$, so that such intervals partition the line).

Let $X^+$ be an independently-IID-marked version of $X$. The idea now is to independently resample the points of $X$ in any such interval as above, while harnessing the internal randomness to simultaneously generate, in a bijective manner, an independent mark for each point. As this operation is finitary and preserves the special points, it will yield a finitary isomorphism from $X$ to $X^+$. We proceed to make this precise.

Let $(x',x]$ be an interval of consecutive special points, i.e., $x'<x$ are two special points having no other special point in between. We condition on the number of points in the interval and their locations, except for that of $x_1$. That is, we condition on $X \cap (x',x] \setminus \{x_1\}$, or equivalently, on $X \cap (x',x_2]$. 
Observe that, given this, the conditional distribution of $x-x_1$ (and hence of $x_1-x_2$) depends only on $\ell=x-x_2$.
In fact, this conditional distribution is simply the conditional distribution of $S$ given that $S+S'=\ell$, where $S$ and $S'$ are independent random variables with distribution $\Pr(T \in \cdot \mid T \in A)$, so that it is absolutely continuous by \cref{lem:continuous-cond-distrib}.
Let $D_\ell$ be a random variable with this distribution. 
Then for every $\ell>0$ and integer $k \ge 0$, since $D_\ell$ is continuous, the isomorphism theorem for probability spaces~\cite[Theorem~3.4.23]{srivastava2008course} yields the existence of a Borel isomorphism (a bimeasurable bijection) $\phi_{\ell,k}$ such that
\[ \phi_{\ell,k}(D_\ell) \eqd (D_\ell, U_1,\dots,U_k) ,\]
where $U_1,\dots,U_k$ are independent random variables, independent also of $D_\ell$, whose distribution is that of a single mark of $X^+$.
Finally, conditioning on the set of special points, and for each such interval $(x',x]$ of consecutive special points, also on the restriction of $X$ to $(x',x_2]$, letting $k$ denote the number of points in $(x',x]$ (which is always two more than the number of points in $(x',x_2]$), we can apply $\phi_{\ell,k}(x-x_1)=(d,u_1,\dots,u_k)$ to obtain the resampled point $x-d$ in place of $x_1$, along with $k$ independent marks which we may assign to the points in $(x',x]$ according to their order (e.g., $u_1$ is assigned to $x$, $u_2$ to $x_1$, and so forth). It is straightforward that this yields a marked point process with the same law as $X^+$, and that this mapping is finitary.
Finally, since the resampling of $x_1$ does not change the special points, and since each $\phi_{\ell,k}$ is a bijection, this mapping is easily seen to be invertible with a finitary inverse, thereby establishing the lemma.
\end{proof}

We remark that while the above resampling procedure preserved the number of points in each interval of consecutive special points (in fact, it preserved much more than this), this is not at all essential; one could just as well resample the entire process in each such interval.

\begin{remark}
Our proof of \cref{lem:marking} involves variations of some ideas from~\cite{ball2005poisson,kosloff2019finitary,soo2019finitary}.
The idea of resampling segments of the point process in order to generate independent labels for the points was used by Kosloff--Soo~\cite[Lemma~5]{kosloff2019finitary} to show that a stationary independently-colored renewal point process whose jump distribution is absolutely continuous and whose skeleton is an irreducible (non-constant) Markov chain is finitarily isomorphic to any of its independently-IID-marked versions (where the points are both colored and marked).
Their construction made use of the colors in the original point process in order to find special points which split the process into independent segments (and hence required the skeleton to be non-constant). On the other hand, Soo proved the analogue for unmarked Poisson point processes~\cite[Proposition~4]{soo2019finitary}, showing that any such process is finitarily isomorphic to any of its independently-IID-marked versions. Soo's construction relied on a sophisticated tool developed by Holroyd--Lyons--Soo~\cite{holroyd2017finitary} for Poisson point processes on $\R^d$ (for arbitrary $d \ge 1$). 
Our construction follows a more direct route, closer to that of~\cite{kosloff2019finitary}, with the main difference being in the definition of the special points.
\end{remark}

We shall require a generalization of \cref{lem:marking}, allowing simply-marked versions rather than just independently-IID-marked versions of renewal point processes. At the same time, we also relax the assumption that the jump distribution is non-singular (while we do not require this relaxation for our main results, we feel that it shines some light on the difficulties in working with singular jump distributions). Indeed, this assumption was only used when appealing to \cref{lem:continuous-cond-distrib}, and the only part of the conclusion of the lemma that was used is that the conditional distribution of $T_1$ given $T_1+T_2$ is continuous. We note that the analogue of \cref{lem:continuous-cond-distrib} for continuous jump distributions is not true. In fact, there exist independent continuous jump distributions $T_1$ and $T_2$ such that $T_1$ and $T_2$ are measurable functions of $T_1+T_2$. For example, construct $T_1$ by taking a uniform random variable in $[0,1]$ and setting its even-indexed bits to 0 in its binary expansion, and similarly construct $T_2$ with the odd-indexed bits set to 0. One can also construct similar situations in which $T_1$ and $T_2$ have the same distribution and the unordered pair $\{T_1,T_2\}$ is a measurable function of $T_1+T_2$. This motivates the following definition.

Say that a jump distribution $T$ is \textbf{continuously-divisible} if there exists an integer $n \ge 1$ and a measurable set $B \subset \R^n$ such that, letting $T_1,T_2,\dots$ be independent copies of $T$, we have that $\Pr((T_1,\dots,T_n) \in B)>0$ and, given $T_1+\cdots+T_n$ and that $(T_1,\dots,T_n) \in B$, the conditional distribution of $(T_1,\dots,T_n)$ is almost surely continuous, and moreover, $B$ is non-overlapping in the sense that $\Pr((T_1,\dots,T_n) \in B, (T_{k+1},\dots,T_{k+n}) \in B)=0$ for all $1 \le k \le n-1$. Note that taking $B=A\times A \times A^c$ with $A$ as in the first line of the proof of \cref{lem:marking}, \cref{lem:continuous-cond-distrib} shows that any non-singular jump distribution is continuously-divisible.


\begin{lemma}\label{lem:marking2}
Let $X$ be a stationary renewal point process whose jump distribution is continuously-divisible. Then $X$ is finitarily isomorphic to any of its simply-marked versions.
\end{lemma}

\begin{proof}
The proof is very similar to that of \cref{lem:marking}, and we only indicate the required changes. To keep things simple, we explain separately the two required modifications; it should then be clear how to make both changes simultaneously.

We first explain how to allow simply-marked versions rather than just independently-IID-marked versions. As before, given an interval $(x',x]$ of consecutive special points, we condition on the restriction of $X$ to $(x',x_2]$. Letting $x'=x_k<\cdots<x_1<x_0=x$ denote the points in $[x',x]$, this is equivalent to conditioning on $k$ and $(x_2,\dots,x_{k-1})$. Recall that a simply-marked renewal point process is described by its jump-mark distribution, a joint distribution $(T,M)$ of the distance to and mark of the next point. Let $((T_i,M_i))_{1 \le i \le k}$ be independent random variables having this distribution. Let $((T'_i,M'_i))_{1 \le i \le k}$ be random variables having the conditional distribution of $((T_i,M_i))_{1 \le i \le k}$ given that $T_1+\cdots+T_k=x-x'$, $T_i=x_i-x_{i-1}$ for all $3 \le i \le k$, and $T_1,T_2 \in A$.
Then $(T'_i)_{1 \le i \le k}$ has the same distribution as the conditional distribution of $(x_1-x_0,\dots,x_k-x_{k-1})$, so that $T'_1$ has the same distribution as $D_\ell$, where $\ell=x-x_2$. In particular, $((T'_i,M'_i))_{1 \le i \le k}$ has a continuous distribution so that the isomorphism theorem yields a Borel isomorphism such that
\[ \phi_{\ell,k}(D_\ell) \eqd (T'_1, M'_1,\dots,M'_k) .\]
Continuing as before, this allows to resample the point $x_1$, while generating marks with the desired distribution for the points $x_0,\dots,x_{k-1}$.

We now explain how to allow continuously-divisible jump distributions rather than just non-singular ones.
For this, we modify the definition of special points. Let $n \ge 1$ and $B \subset \R^n$ be as in the definition of continuously-divisible. Given a point $x$ of $X$, let $x_n<\cdots<x_2<x_1<x_0=x$ denote the $n$ nearest points to its left.
Call the point $x$ special if $(x_0-x_1,x_1-x_2,\dots,x_{n-1}-x_n) \in B$. Clearly, special points exist, and if $x$ is special, then $x_1,\dots,x_{n-1}$ are not special. Note that if $(x',x]$ is an interval of consecutive special points, then given $X \cap (x',x_n]$, the conditional distribution of $(x_1,\dots,x_{n-1})$ is continuous.
The proof then proceeds along the same lines as before: conditioning on the set of special points, for each interval $(x,x']$ of consecutive special points, we apply a Borel isomorphism in order to resample the process in that interval (or just the points $x_1,\dots,x_{n-1}$), while simultaneously generating independent marks for the points (independent also of the position of the points themselves), with the distribution of each mark chosen according to the distance of its point to the nearest point to its left.
\end{proof}

\cref{lem:marking2} may be equivalently stated as saying that any stationary simply-marked renewal point process whose jump distribution is continuously-divisible is finitarily isomorphic to its underlying unmarked point process.
The following is a version of this for finitarily-regenerative processes.

\begin{lemma}\label{lem:regenerative}
Let $X$ be a stationary finitarily-regenerative process and let $Y$ be a stationary renewal point process whose jump distribution is continuously-divisible. Suppose that $Y$ is finitarily-regenerative for $X$. Then $X$ is finitarily isomorphic to $Y$.
\end{lemma}
\begin{proof}
Construct a marked point process $Y^+$ from $Y$ by marking a point $b$ of $Y$ by the excursion of $X$ from $a$ to $b$, namely, $(X_{t+a})_{0 \le t < b-a}$, where $a$ is the nearest point of $Y$ to the left of $b$.
Since $Y$ is a finitary factor of $X$, it is clear that $Y^+$ is finitarily isomorphic to $X$.
Since the paths of $X$ are right-continuous with left-hand limits in a Polish space, $(X_{t+a})_{0 \le t < b-a}$ takes values in a Polish space, and since $Y$ is finitarily-regenerative for $X$, it follows from the definitions that $Y^+$ is a simply-marked version of $Y$. Thus, \cref{lem:marking2} yields that $X$ is finitarily isomorphic to $Y$.
\end{proof}

Let us show how these, together with \cref{thm:RP2}, imply \cref{thm:simply-marked-RP} and \cref{thm:regenerative}.

\begin{proof}[Proof of \cref{thm:simply-marked-RP}]
Let $X$ be a stationary renewal point process whose jump distribution has exponential tails and a non-singular convolution power, and let $X^+$ be a simply-marked version of $X$. By \cref{thm:RP2}, $X$ is finitarily isomorphic to a Poisson point process, and by \cref{lem:marking2}, it is also finitarily isomorphic to $X^+$, showing that $X^+$ is finitarily isomorphic to a Poisson point process.
\end{proof}

\begin{proof}[Proof of \cref{thm:regenerative}]
Let $X$ be a stationary finitarily-regenerative process having an associated jump distribution $T$ that has exponential tails and a non-singular convolution power. By definition, there exists a renewal point process $Y$ with jump distribution $T$, which is finitarily-regenerative for $X$. By \cref{thm:RP2}, $Y$ is finitarily isomorphic to a Poisson point process, and by \cref{lem:regenerative}, it is also finitarily isomorphic to $X$, showing that $X$ is finitarily isomorphic to a Poisson point process.
\end{proof}

\section{Renewal point processes with regular jump distributions}
\label{sec:renewal-with-regular-jd}

The first goal of this section is to prove \cref{thm:RP} for a very particular class of renewal point processes.
A \textbf{simple-selection point process} is any point process $X$ which can be obtained as follows. Let $\Lambda$ be an independently-IID-colored Poisson point process with two colors (the two colors have positive probability, but need not have equal probability). Suppose that one of the two colors is red, and fix $t \ge 0$. The points of $X$ are those red points $x$ of $\Lambda$ for which $\Lambda$ contains no point (of either color) in $(x-t,x)$.
We note that $X$ is a stationary renewal point process.

\begin{lemma}\label{lem:simple-selection}
Any simple-selection point process is finitarily isomorphic to any Poisson point process.
In particular, any two Poisson point processes are finitarily isomorphic.
\end{lemma}

\begin{proof}
Let $X$ be a simple-selection point process and let $\Lambda$ and $t$ be as above. We begin by showing that $X$ is finitarily isomorphic to \emph{some} Poisson point process. Specifically, we show that $X \fiso \Lambda'$, where $\Lambda'$ is the uncolored Poisson point process underlying $\Lambda$. Since $\Lambda' \fiso \Lambda$ by \cref{lem:marking}, it suffices to show that $X \fiso \Lambda$. Call a point of $\Lambda$ special if it is red and there is no other point at distance less than $t$ to its left. Clearly, the special-point process $S$ has the same law as $X$, and is a finitary factor of $\Lambda$ (in fact, it is a block factor). To see that the two are finitarily isomorphic, note that, given $S$, the restrictions of $\Lambda$ to the (half-open half-closed) intervals between two consecutive special points are conditionally independent. Moreover, the conditional distribution on any such interval $(a,b]$ depends only on its length $b-a$. Thus, $S$ is regenerative for $\Lambda$. Since the jump distribution of $S$ is easily seen to be absolutely continuous with exponential tails, \cref{lem:regenerative} yields that $S$ is finitarily isomorphic to $\Lambda$. This completes the proof that $X \fiso \Lambda'$.

To obtain the lemma, it remains to show that any two Poisson point processes are finitarily isomorphic, as this will also show that $X$ is finitarily isomorphic to \emph{any} Poisson point process.
Using the notation above, let us say that $X$ is defined through $(\Lambda,\Lambda',t)$. By what we have just shown, any simple-selection point process defined through $(\Pi,\Pi',s)$ is finitarily isomorphic to $\Pi'$. Observe that a Poisson point process of intensity $\lambda$ can be obtained as a simple-selection point process defined through $(\Pi,\Pi',0)$, where $\Pi'$ is a Poisson point process of any intensity $\lambda'>\lambda$. Indeed, construct $\Pi$ by independently coloring the points of $\Pi'$ red with probability $\lambda/\lambda'$, and blue otherwise. The simple-selection point process obtained (which simply consists of all red points of $\Pi$) is a Poisson point process of intensity $\lambda$. This shows that the two Poisson point processes of intensities $\lambda$ and $\lambda'$ are finitarily isomorphic.
\end{proof}

The main goal of this section is to prove \cref{thm:RP} under a regularity assumption on the jump distribution. To do so, we first isolate the following simple corollary. A \textbf{single-mark marginal} of a marked point process is the unmarked point process consisting of all points having a fixed mark.

\begin{cor}\label{cor:multi-RP}
Let $X$ be a stationary renewal point process with a non-singular jump distribution. Suppose that some simply-marked version of $X$ has some single-mark marginal that is a simple-selection point process. Then $X$ is finitarily isomorphic to any Poisson point process.
\end{cor}

\begin{proof}
Let $X^+$ be a simply-marked version of $X$ having some fixed-mark marginal $X'$ that is a simple-selection point process. By \cref{lem:simple-selection}, $X'$ is finitarily isomorphic to any Poisson point process. Clearly, $X'$ is a finitary factor of $X^+$, and since $X^+$ is simply-marked, it is straightforward that $X'$ is regenerative for $X^+$. \cref{lem:regenerative} thus implies that $X^+ \fiso X'$. Finally, since $X \fiso X^+$ by \cref{lem:marking2}, we conclude that $X$ is finitarily isomorphic to any Poisson point process.
\end{proof}

We are now ready to prove the following special case of \cref{thm:RP}.

\begin{prop}\label{prop:RP-special-case}
Any Poisson point process is finitarily isomorphic to any stationary renewal point process $X$ whose jump distribution $T$ is unbounded and absolutely continuous with a density $f$ satisfying
\begin{equation}\label{eq:regular-density}
 \liminf_{t \to \infty} \frac{f(t)}{\Pr(T>t)} > 0 \qquad\text{and}\qquad \sup_{t \ge 0} \frac{f(t)}{\Pr(T>t)} < \infty .
\end{equation}
\end{prop}

\begin{proof}
Our strategy is to show that $X$ satisfies the assumptions of \cref{cor:multi-RP}. To do so, we define a simply-marked version $X^+$ of $X$, with the marks taking one of two possible values called red and blue, such that the red-color marginal of $Z$ is a simple-selection point process. Once we have done this, the proposition will immediately follow from \cref{cor:multi-RP}.

Let us first explain the idea of the proof in the simpler situation in which
\[ \lambda_c := \inf_{t \ge 0} \frac{f(t)}{\Pr(T>t)} > 0 .\]
To define the simply-marked point process $X^+$, it suffices to describe its jump-marked distribution, i.e., the joint distribution of $(T',C)$, the distance to and color of the subsequent point, as this completely describes any stationary simply-marked point process.
Let $T_1$ be an exponential random variable with rate $\lambda<\lambda_c$. We claim that there exists a (unique) random variable $T_2$, independent of $T_1$, such that $\min\{T_1,T_2\}$ has the same distribution as $T$. Indeed, such a variable is defined by
\[ \Pr(T_2 > t) = \Pr(T > t) \cdot e^{\lambda t} \qquad\text{for all }t>0 .\]
To see that this defines a random variable, we must only check that $g(t):=\Pr(T > t) \cdot e^{\lambda t}$ is non-increasing (it will not cause us trouble if $T_2=\infty$ with positive probability, though one may check that this cannot happen). Since $g$ is absolutely continuous, it suffices to show that its derivative is non-positive almost everywhere. Indeed, $g'(t) = (\lambda \Pr(T>t) - f(t))e^{\lambda t}$, so that this follows from the fact that $\lambda<\lambda_c$.
We also mention that no such variable $T_2$ exists for any $\lambda>\lambda_c$.
We now define
\[ (T',C) := \begin{cases}
 (T_1,\text{red}) &\text{if }T_1 \le T_2 \\
 (T_2,\text{blue}) &\text{if }T_1 > T_2
\end{cases}.\]
Then $T'$ has the same distribution as $T$, so that $X^+$ is indeed a simply-marked version of $X$ (i.e., the uncolored point process underlying $X^+$ has precisely the same law as $X$). One may imagine that there are two independent clocks governing this simply-marked renewal point process: a red clock which rings after time $T_1$ and a blue clock which rings after time $T_2$. When a clock rings, a point of that color is placed and both clocks are reset. In fact, since the exponential distribution is memoryless, there is no need to reset the red clock when the blue clock rings. The red clock is thus unaffected by the blue clock. It follows that the red-color marginal of $X^+$ is simply a Poisson point process of intensity $\lambda$. Thus, \cref{cor:multi-RP} yields that $X$ is finitarily isomorphic to any Poisson point process.

\smallskip

Let us now return to the general case. The simply-marked point process $X^+$ can be defined in the same manner as above, with the distributions of $T_1$ and $T_2$ chosen differently.
However, since $\lambda_c$ may be zero, we cannot take $T_1$ to be exponential, and thus should not hope that the red-color marginal of $X^+$ will be a Poisson point process as before. Instead, we aim to have the red-color marginal be a simple-selection point process. Rather than defining $T_1$ and $T_2$ explicitly as in the simplified situation above, we define $X^+$ here in a different way which makes the construction more illuminating.

Let $\Pi$ be a Poisson point process of unit intensity on $\R \times \R^+$. In practice, our construction will only depend on the restriction of $\Pi$ to $\R \times (0,\lambda')$ for a large fixed $\lambda'$.
For $h>0$, let $\Pi_h$ be the projection of $\Pi$ from $\R \times (0,h)$ onto $\R$, so that $\Pi_h$ is a Poisson point process of intensity $h$ on $\R$.
For $x \in \R$, we define
\[ T(x) := \inf \{ t>0 : x+t \in \Pi_{f(t)/\Pr(T>t)} \} .\]
We remark that $f(t)/\Pr(T>t)$ may be thought of as the ``instantaneous rate'' with which points appear in the renewal point process $X$, given that time $t$ has elapsed since the last point appeared (see \cref{rem:reversed-process}). To substantiate this, let us check that, for any fixed $x$, $T(x)$ has the same distribution as $T$ (see \cref{fig:pp1}). By translation-invariance, it suffices to check this for $x=0$. Indeed,
\begin{align*}
 \Pr(T(0)>t)
  &= \Pr\big(s \notin \Pi_{f(s)/\Pr(T>s)}\text{ for all }0<s<t\big) \\
  &= \Pr\big(\Pi\text{ has no point in }\big\{(s,h) : 0<s<t,~0<h<f(s)/\Pr(T>s)\big\}\big) \\
  &= \exp\left[- \text{Leb}\big\{(s,h) : 0<s<t,~0<h<f(s)/\Pr(T>s)\big\}\right] \\
  &= \exp\left[- \int_0^t \frac{f(s)}{\Pr(T>s)}ds \right] = \exp\left[- \int_{\Pr(T>t)}^1 \frac{1}{u} du \right] = \Pr(T>t).
\end{align*}
Thus, we may regard $T(x)$ as the distance to the next point of $X$ after $x$, given that there is a point of $X$ at~$x$. To make this more precise, define
\[ N_0(x):=x, \qquad\text{and}\qquad N_{n+1}(x) := N_n(x)+T(N_n(x)) \qquad\text{for }n \ge 0 .\]
Let $P_x$ be the point process on $[x,\infty)$ defined by
\[ P_x := \{ x, N_1(x), N_2(x), \dots \} .\]
See \cref{fig:pp2} for an illustration of this point process.
Note that, for any fixed $x$, $T(x)$ is a stopping time with respect to the filtration $(\cF_{x+t})_{t \ge 0}$, where $\cF_t$ is the $\sigma$-algebra generated by the restriction of $\Pi$ to $(-\infty,t] \times \R^+$.
Similarly, for any fixed $x$ and $n$, $N_n(x)-x$ is a stopping time with respect to the same filtration.
Using this and the strong Markov property of the Poisson point process, it is easy to check that, for any fixed $x$, $P_x$ is a one-sided renewal point process on $[x,\infty)$ with jump distribution $T$ and a first point at $x$. Note also that once two such processes have a common point, they agree from that point onward, i.e.,
\begin{equation}\label{eq:coalesce}
\text{$P_{x_1}$ and $P_{x_2}$ coincide on $[x,\infty)$ for any $x$ that is a common point of both $P_{x_1}$ and $P_{x_2}$} .
\end{equation}

\begin{figure}
\includegraphics[scale=0.65,trim={0cm 0cm 1cm 0cm},clip]{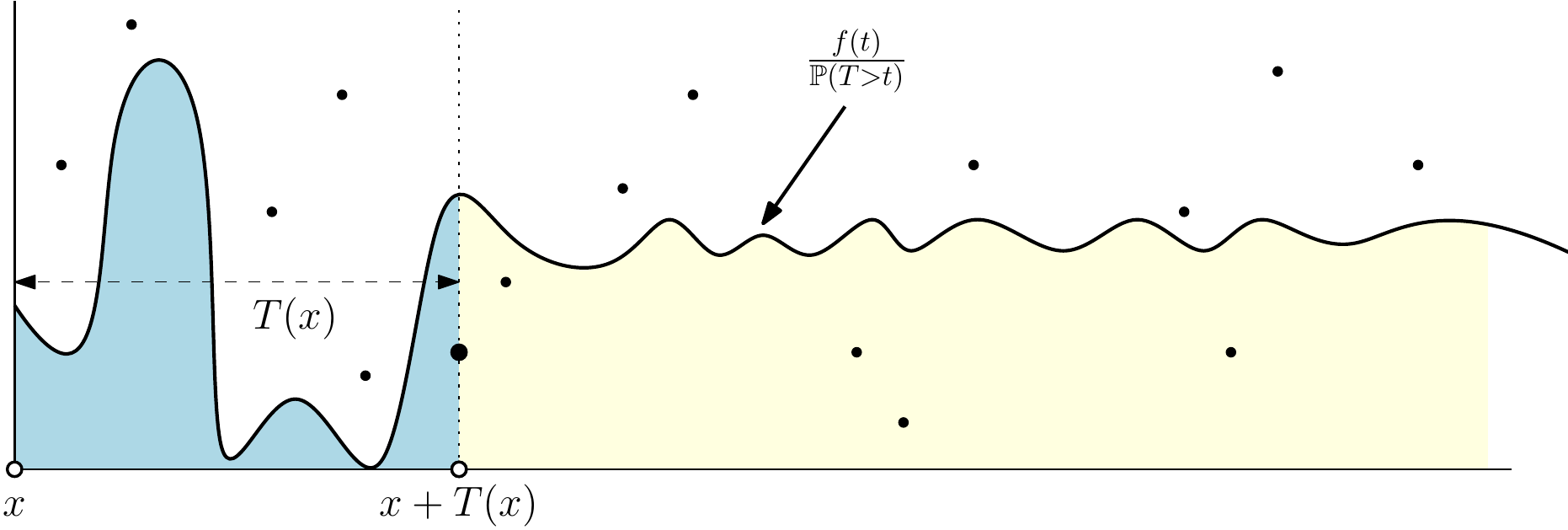}
\caption{One may find many copies of the jump distribution embedded within a two-dimensional Poisson point process. Namely, for each $x$, the time $T(x)$ until the arrival of the first point which lies below the graph of the function $f(t)/\Pr(T>t)$ (shifted to start at $x$) has the same law as $T$.}
\label{fig:pp1}
\end{figure}
\begin{figure}
\includegraphics[scale=0.65,trim={0cm 0cm 1cm 0cm},clip]{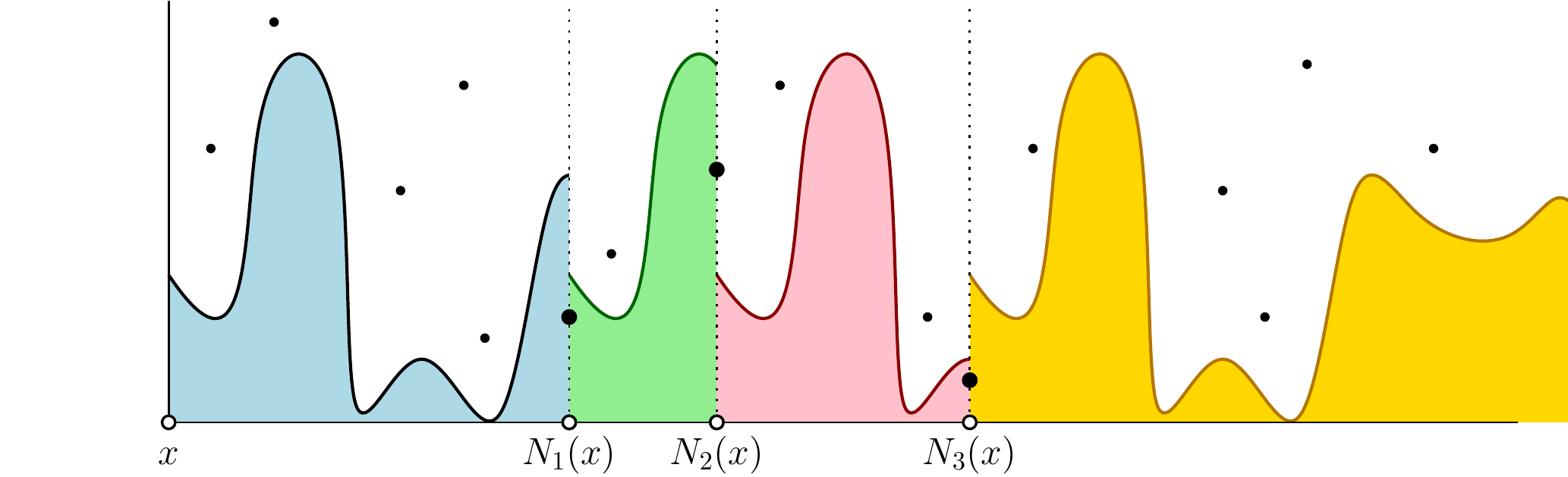}%
\caption{Constructing the one-sided renewal point process $P_x$. Starting at a given point $x$, one iteratively constructs $P_x$ by finding the distance to the next point according to the procedure in \cref{fig:pp1}. Starting at a different point $x'$ yields a different point process $P_{x'}$. If two such point processes happen to agree on a point, then they agree on all subsequent points.}
\label{fig:pp2}
\end{figure}

\begin{figure}
\includegraphics[scale=0.65,trim={0cm 0cm 1cm 0cm},clip]{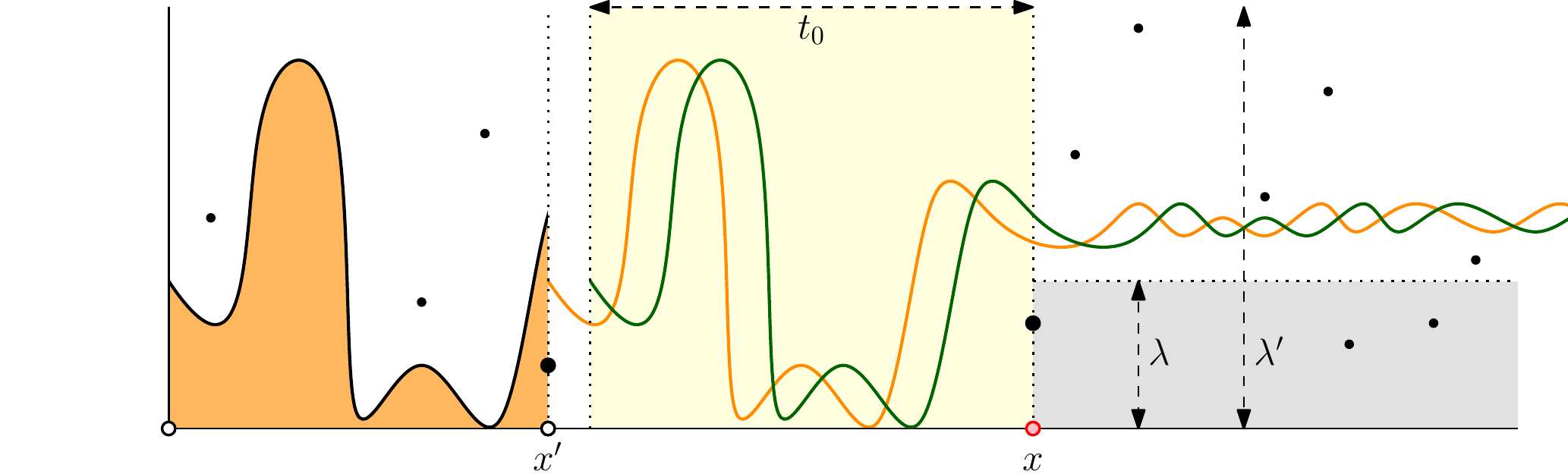}%
\caption{Constructing the stationary point process $\tilde\Pi$. Since $[t_0,\infty) \times [0,\lambda]$ lies entirely below the graph of $f(t)/\Pr(T>t)$, it also lies entirely below any translation of this graph to the left. This means that if $\Pi$ has a point of height at most $\lambda$ at time $x$, then starting from any point $x' \le x-t_0$, the next point in $P_{x'}$ cannot be later than $x$. If $\Pi$ also has no points of height less than $\lambda'$ between times $x-t_0$ and $x$, then $P_{x'}$ cannot have a point strictly between $x-t_0$ and $x$. It follows that $x$ is a point of every such $P_{x'}$. Such points make up $\tilde\Pi$ and are colored red in $X^+$.}
\label{fig:pp3}
\end{figure}

This last observation raises the possibility of obtaining a limiting process as $x \to -\infty$. Indeed, our next goal is to show that, in a precise sense, $P_x$ locally stabilizes as $x \to -\infty$, and that this can be detected in a finitary manner.
By the assumption~\eqref{eq:regular-density} on the jump distribution $T$, there exist $t_0 \ge 0$ and $\lambda'>\lambda>0$ such that
\begin{equation*}\label{eq:bounded-local-rate}
\frac{f(t)}{\Pr(T>t)} > \lambda \qquad\text{for all } t \ge t_0, \qquad\qquad \frac{f(t)}{\Pr(T>t)} < \lambda' \qquad\text{for all } t \ge 0 .
\end{equation*}
Let $\tilde{\Pi}$ be the stationary point process defined by
\[ \tilde{\Pi} := \{ x \in \Pi_\lambda : \Pi_{\lambda'}\text{ has no point in }(x-t_0,x) \} .\]
See \cref{fig:pp3} for an illustration.
Note that
\[ \tilde{\Pi} \cap [x,\infty) \subset P_{x-t_0} \qquad\text{for all }x \in \R .\]
In particular, recalling~\eqref{eq:coalesce}, it follows that
\[ \text{$P_{x_1}$ and $P_{x_2}$ coincide on $[x,\infty)$ for any $x$ in $\tilde{\Pi}$ such that $x \ge \max\{x_1,x_2\}+t_0$} .\]
Since $\tilde{\Pi}$ is almost surely unbounded in $(-\infty,0]$, we deduce that
\[ \text{almost surely:} \qquad \substack{\text{\normalsize for any $x \in \R$ there exists $x' \le x$ such that}\\\text{\normalsize $P_{x_1}$ and $P_{x_2}$ coincide on $[x,\infty)$ for any $x_1,x_2 \le x'-t_0$.}} \]
Moreover, $x'$ can be chosen finitarily by taking it to be the first point of $\tilde\Pi$ to the left of $x$.
Let $X$ be the limiting point process $\lim_{x \to -\infty} P_x$. Note that $X$ has a very concrete description, namely, its restriction to an interval $[a,b]$ coincides with the restriction of $P_{x-t_0}$ to the same interval, where $x$ is any point to the left of $a$ such that $\tilde{\Pi}$ has a point in $[x,a]$. It is immediate from this that $X$ has the desired law, i.e., it is a stationary renewal point process with jump distribution $T$.

The above already shows that $X$ is a finitary factor of $\Pi$, a Poisson point process on $\R \times \R^+$ (which is finitarily isomorphic to a Poisson point process on $\R$; but this will not be used directly).
To see that it is finitarily isomorphic to a Poisson point process (on $\R$), we color the points of $X$ in two colors: points of $\tilde{\Pi}$ (all of which are also points of $X$) are colored red, and the rest are colored blue.
Let $X^+$ denote this colored point process.

We claim that $X^+$ is a simply-marked version of $X$. To see this, note first that the restriction of $X^+$ to $(-\infty,s]$ is determined by the restriction of $\Pi$ to $(-\infty,s] \times \R^+$. Next, note that if $X^+$ has a point at $s$ (of either color), then $\tilde\Pi$ cannot have a point in $(s,s+t_0)$. This is because a point of $X$ at $s$ implies that $s \in \Pi_{\lambda'}$, while a point of $\tilde\Pi$ at $x$ implies that $\Pi_{\lambda'}$ has no point in $(x-t_0,x)$. Next, note that the restriction of $\tilde\Pi$ to $[s+t_0,\infty)$ is determined by the restriction of $\Pi$ to $(s,\infty) \times \R^+$. Finally, note that $(T(x))_{x \ge s}$, and hence also $P_s$, is determined by the restriction of $\Pi$ to $(s,\infty) \times \R^+$. Together this shows that the restriction of $X^+$ to $(s,\infty)$ does not depend on its restriction to $(-\infty,s]$ whenever it has a point (of either color) at~$s$. To be more precise with this last statement, we mean that, for any fixed $s$, if $S$ denotes the smallest point of $\Pi$ in $[s,\infty)$, then the restriction of $\theta_S(X^+)$ to $(0,\infty)$ is independent of $\cF_S$. Indeed, this follows from the strong Markov property of the Poisson point process and the fact that $X^+$ is adapted to the filtration $(\cF_t)$.
Thus, $X^+$ is simply-marked.

Observe that $\tilde{\Pi}$ is a simple-selection point process. Since the red-color marginal of $X^+$ is $\tilde{\Pi}$, we see that the assumptions of \cref{cor:multi-RP} are satisfied by $X$, and thus conclude that $X$ is finitarily isomorphic to any Poisson point process.
\end{proof}

\begin{remark}\label{rem:reversed-process}
Any stationary renewal point process $X$ has a naturally associated continuous-time Markov process $Y$ consisting of (right-continuous) piecewise linear paths of slope $-1$ with jump discontinuities at the points of $X$. That is, if $a$ and $b$ are two consecutive points of $X$, then $Y_{b-t} = t$ for $t \in (0,b-a]$. Note that the value of $Y$ at some point $s$ represents the \emph{time until the next point}, so that $Y$ is deterministic at all times except when `hitting' 0, at which time it instantaneously jumps. In this sense, the randomness of the jump length between a point of $X$ and its consecutive point is packed entirely into a single point of time in $Y$; namely, the jump length in $X$ between a point $a$ and its consecutive point is precisely the height of the jump discontinuity at $a$ in $Y$.

Consider the (right-continuous version of the) time-reversed process $Y'$. In this process, the value at some time $s$ can be thought of as representing the \emph{time since the previous point}, so that, unlike $Y$, the process is not so deterministic. In some sense, the randomness of the jump length is spread-out across the interval between the jump discontinuities. Indeed, if $Y'_s=t$, then the probability that $Y'$ does not jump to 0 in the next $\epsilon$ time (i.e., that $Y'_{s+\delta}=t+\delta$ for all $0 \le \delta \le \epsilon$) is $\Pr(T>t+\epsilon \mid T>t)$. Thus, if $T$ is absolutely continuous with a density $f$, then the ``instantaneous rate'' at which $Y'$ jumps to 0 is $f(t)/\Pr(T>t)$, where $t$ is the value of $Y'$ at that instant. With this in mind, the idea used in the proof of \cref{prop:RP-special-case} may be thought of as a form of coupling-from-the-past for $Y'$.

This idea was motivated by a construction by Angel and the author~\cite{angel2019markov}, who showed that any discrete-time renewal process whose jump distribution has exponential tails and is not supported in a proper subgroup of $\Z$ is a finitary factor of an \iid\ process.
\end{remark}

%

\section{Regularization of jump distributions}
\label{sec:regularize-jd}

In this section, we show how one may reduce the problem for a renewal point process with jump distribution $T$ to a renewal point process with a modified jump distribution, which is more regular than $T$. To define this modified jump distribution, let $T_1,T_2,\dots$ be independent copies of $T$, let $N$ be a non-negative integer-valued random variable, and define
\[ T^*_N := T_1 + \cdots + T_N .\]
We shall be interested in two particular situations.
In the first situation, $N$ is independent of $(T_n)_n$ and is a geometric random variable with success probability $p \in (0,1)$, by which we mean that $\Pr(N=n) = p(1-p)^{n-1}$ for $n \ge 1$. In this case, we write $T^*_{\text{Geom}(p)}$ for $T^*_N$, and $T^*_{\text{Geom}(p)-1}$ for $T^*_{N-1}$.
In the second situation, $N$ is a stopping time with respect to $(T_n)_n$, obtained as the first $k$-consecutive hitting time of a set $A$, for a measurable set $A \subset \R$ and an integer $k \ge 1$. Precisely, 
\[ N = N_{k,A} := \min\{ n \ge k : T_n,T_{n-1},\dots,T_{n-k+1} \in A \} .\]
In this case, we write $T^*_{\text{$k$-hit}(A)}$ for $T^*_N$.



The following lemma allows us to replace the original jump distribution $T$ with the modified jump distribution $T^*_{\text{$k$-hit}(A)}$ when proving that a stationary renewal point process is finitarily isomorphic to a Poisson point process. Recall the definition of continuously-divisible from \cref{sec:simply-marked} and that non-singular jump distributions are continuously-divisible.

\begin{lemma}\label{lem:stopping-time}
Let $T$ be a jump distribution, let $A \subset \R$ be a measurable set such that $\Pr(T \in A)>0$, and let $k \ge 1$. Suppose that $T^*_{\text{$k$-hit}(A)}$ is continuously-divisible. Then the stationary renewal point process with jump distribution $T$ is finitarily isomorphic to the stationary renewal point process with jump distribution $T^*_{\text{$k$-hit}(A)}$.
\end{lemma}

\begin{proof}
Let $X$ and $Y$ denote the stationary renewal point processes with jump distributions $T$ and $T^*_{\text{$k$-hit}(A)}$, respectively. We may assume that $\Pr(T \in A)<1$, as otherwise $X$ and $Y$ have the same law.
Given a point $x$ of $X$, let $x=x_0>x_1>x_2>\cdots$ denote the points to its left in decreasing order, and let $L(x)$ denote the maximal $m \ge 0$ such that $x_{i-1}-x_i \in A$ for all $1 \le i \le m$.
Let $S$ denote the point process consisting of all points $x$ of $X$ for which $L(x) \in \{k,2k,3k,\dots\}$.
Observe that $S$ is a finitary factor of $X$. Also, it is straightforward that $S$ has the same law as $Y$ and that it is regenerative for $X$. Thus, \cref{lem:regenerative} yields that $S$ is finitarily isomorphic to $X$.
\end{proof}


For a jump distribution $T$ with exponential tails, define
\[ \cL(T) := \sup \{ \E[e^{\lambda T}] : \lambda>0\text{ such that }\E[e^{\lambda T}]<\infty \}.\]
We say that $T$ is \textbf{non-arithmetic} if it is not supported in $a\Z$ for any $a>0$.

\begin{lemma}\label{lem:exp-geom}
Let $T$ be a non-arithmetic jump distribution with exponential tails.
Let $0<p<1-\frac1{\cL(T)}$. Then there exist $b,c>0$ such that
\[ \Pr(T^*_{\text{Geom}(p)-1} > t) = c e^{-bt} (1+o(1)) \qquad\text{as }t \to \infty .\]
\end{lemma}

We remark that the same holds for $T^*_{\text{Geom}(p)}$ (with the same $b$ and a different $c$). Indeed, this follows from the simple observation that the law of $T^*_{\text{Geom}(p)}$ is the same as the conditional law of $T^*_{\text{Geom}(p)-1}$ given that it is positive.

\begin{proof}
Let $F(\lambda) := \E[e^{\lambda T}]$ be the moment generating function of $T$. Since $T$ has exponential tails, $a := \sup \{\lambda \ge 0 : F(\lambda)<\infty\}$ is strictly positive (perhaps $\infty$). 
Note that $F$ is strictly increasing and continuous on $(-\infty,a)$, and that $F(0)=1$ and $F(\lambda) \to \cL(T)$ as $\lambda \to a^-$. Thus, for any $p$ as in the lemma, there exists a unique $b \in (0,a)$ such that $(1-p)F(b) = 1$.

The claimed estimate can now be obtained by exploiting a relation with renewal theory. Indeed, $T^*_{\text{Geom}(p)-1}$ is the terminating time of a transient renewal process, and the lemma can be obtained as a corollary of the key renewal theorem; see~\cite[XI.6, Theorem~2 and (6.16)]{feller2008introduction} or \cite[Chapter~3.11]{resnick2013adventures}.
\end{proof}

\begin{remark}
Besides the connection with renewal theory, random geometric sums such as $T^*_{\text{Geom}(p)}$ are also related to Cram\'er's theorem, R\'enyi's theorem and certain Tauberian theorems (see, e.g., \cite{feller2008introduction,kalashnikov2013geometric}).
For example, rather than using the key renewal theorem in the proof above, one could alternatively appeal to a Tauberian theorem (see, e.g., \cite[Theorem~2.1]{beare2017geometrically}), applied to the moment generating function of $T^*_{\text{Geom}(p)}$, to obtain \cref{lem:exp-geom}.
We will eventually apply the lemma to an absolutely continuous jump distribution, in which case, more precise estimates for the tail probabilities of $T^*_{\text{Geom}(p)}$ are available. For example, \cite[Theorem~2]{blanchet2007uniform} shows that when $p$ is sufficiently small, the asymptotics $ce^{-bt}(1+o(1))$ can be sharpened to $ce^{-bt} + o(e^{-b't})$ for some $b'>b$. We mention that the analogous statement for arithmetic jump distributions can be shown using basic complex analysis; see~\cite[Lemma~6]{angel2019markov}.
\end{remark}

The following lemma, together with \cref{lem:stopping-time}, will allow us to successively improve the regularity properties of the jump distribution, eventually ensuring that it satisfies the assumptions of \cref{prop:RP-special-case}. Together this leads to a proof of \cref{thm:RP} and of the if part of \cref{thm:RP2}.

\newlist{lemmaenum}{enumerate}{1}
\setlist[lemmaenum]{label=(\arabic*), ref=\thethm(\arabic*)}
\crefalias{lemmaenumi}{lemma}

\begin{lemma}\label{lem:regularize-T}
Let $T$ be a jump distribution.
\begin{lemmaenum}
 \item\label{lem:regularize-T-exp} If $T$ has exponential tails, then $T^*_{\text{$k$-hit}(A)}$ has exponential tails for any $A$ and $k$.
 \item\label{lem:regularize-T-power} If $T$ has a non-singular $k$-th convolution power, then $T^*_{\text{$k$-hit}(A)}$ is non-singular for some $A$.
 \item\label{lem:regularize-T-nonsing} If $T$ is non-singular, then there exists $A$ such that $T^*_{\text{2-hit}(A)}$ is absolutely continuous with a continuous and bounded density.
 \item\label{lem:regularize-T-abscont} If $T$ is absolutely continuous with a continuous density and exponential tails, then there exists $A$ such that $T^*_{\text{1-hit}(A)}$ is absolutely continuous with a density $f$ satisfying
\begin{equation}\label{eq:1-hit-regularity}
  \liminf_{t \to \infty} \frac{f(t)}{\Pr(T^*_{\text{1-hit}(A)}>t)} > 0 \qquad\text{and}\qquad \sup_{t \ge 0} \frac{f(t)}{\Pr(T^*_{\text{1-hit}(A)}>t)} < \infty .
\end{equation}
\end{lemmaenum}
\end{lemma}

In the above, $A$ is always a measurable subset of $\R$ such that $0<\Pr(T \in A)<1$. In the proof below, we let $T|_A$ denote a random variable having the conditional distribution of $T$ given that $T \in A$. 

\begin{proof}

\smallskip\textbf{(1)}
Noting that $N_{k,A}$ has exponential tails, this follows from the standard fact that the sum of a random number $N$ of \iid\ random variables having exponential tails also has exponential tails if $N$ does (regardless of whether $N$ is independent of the summands).

\smallskip\textbf{(2)}
Let $\sigma$ be the $k$-th convolution power of $T$, and let $\epsilon \in (0,1]$ be the mass of the absolutely continuous part of $\sigma$.
Let $A \subset \R$ be any measurable set such that $p := \Pr(T \in A)>\sqrt[k]{1-\epsilon}$. We claim that $T|_A$ is non-singular. Indeed, $\sigma$ can be written as $p^k \mu + (1-p^k)\nu$, where $\mu$ is the $k$-th convolution power of $T|_A$ and $\nu$ is some probability measure. Since $1-p^k < \epsilon$, we deduce that $\mu$ is non-singular.
Observe that $T^*_{\text{$k$-hit}(A)}$ has the law of $S+S'$, where $S$ and $S'$ are independent random variables, with $S'$ distributed according to $\mu$. Thus, since $S'$ is non-singular, so is $T^*_{\text{$k$-hit}(A)}$.

\smallskip\textbf{(3)}
Let $B \subset \R$ be a Lebesgue-null set supporting the singular part of $T$.
Let $g$ be the density of the absolutely continuous part of $T$. By assumption, $\Pr(T \notin B) = \int_\R g > 0$.
Let $d>0$ be such that both $\{g>0\} \cap (0,d)$ and $\{g>0\} \cap (d,\infty)$ have positive Lebesgue measure.
Let $\lambda>0$ be such that $\{ 0<g \le \lambda \} \cap (0,d)$ has positive Lebesgue measure.
Set
\[ A := \{ g \le \lambda \} \cap (0,d) \setminus B \]
and note $\Pr(T \in A) = \int_A g \in (0,1)$.

Observe that $T^*_{\text{2-hit}(A)}$ has the law of $S_0+S_1+S_2$, where $S_0$, $S_1$ and $S_2$ are independent random variables, with $S_1$ and $S_2$ having the law of $T|_A$. Since $T|_A$ is absolutely continuous with a bounded density, it follows that $S_1+S_2$ is absolutely continuous with a continuous and bounded density. Since the convolution of a continuous and bounded function with a probability measure is again such a function, $S_0+S_1+S_2$ is also absolutely continuous with a continuous and bounded density.

\smallskip\textbf{(4)}
Let $g$ be the continuous density of $T$. Let $d>0$ be such that $g(d)>0$. By continuity, there exists $d'>d$ and $\lambda' \ge \lambda>0$ such that $A:=(d,d')$ satisfies $A \subset \{ \lambda \le g \le \lambda' \}$.
Note that $p := \Pr(T \in A) = \int_A g > 0$.
Note also that
\[ \frac{\E[e^{\gamma T}] - pe^{\gamma d'}}{1-p} \le \E[e^{\gamma T|_{A^c}}] \le \frac{\E[e^{\gamma T}]}{1-p} \qquad\text{for all }\gamma \ge 0 .\]
Since $p \to 0$ as $d' \to d$, it follows from this that $\cL(T|_{A^c}) \to \cL(T)$ as $d' \to d$. Thus, by decreasing $d'$ if necessary, we may ensure that $p<1-\cL(T|_{A^c})^{-1}$.

Observe that $T^*_{\text{1-hit}(A)}$ has the same law as $S + S'$, where $S$ and $S'$ are independent random variables having the laws of $(T|_{A^c})^*_{\text{Geom}(p)-1}$ and $T|_A$, respectively. Since $S'$ is absolutely continuous with a bounded density, $S+S'$ is also absolutely continuous with a bounded density $f$. Also, $S$ has no atoms apart from the one at 0.

It remains to establish~\eqref{eq:1-hit-regularity}.
Note first that we may replace the supremum in~\eqref{eq:1-hit-regularity} with a limsup, as the finiteness of the supremum will then follow from the fact that $f$ is bounded. Thus, we now aim to estimate $f(t)$ for large $t$. To do this, we proceed to estimate the probability that $S+S'$ belongs to the interval $(t,t+\epsilon)$ for small $\epsilon>0$. Observe that $S'$ has density between $\lambda/p$ and $\lambda'/p$ on $(d,d')$ and density zero elsewhere. It follows that, almost surely,
\[
\Pr(t<S + S'<t+\epsilon \mid S)
\begin{cases}
 \le \lambda'\epsilon/p & \\
 \ge \lambda\epsilon/p &\quad\text{on the event that }S \in (t-d'+\epsilon,t-d) \\
= 0 &\quad\text{on the event that }S \notin (t-d',t-d+\epsilon)
\end{cases} .\]
Taking expectation, we get that
\[ \tfrac{\lambda\epsilon}{p} \cdot \Pr(t-d'+\epsilon<S<t-d) \le \Pr(t<S+S'<t+\epsilon) \le \tfrac{\lambda'\epsilon}{p} \cdot \Pr(t-d'<S<t-d+\epsilon) .\]
Thus, taking $\epsilon \to 0$, we get that
\[ \frac{\lambda}{p} \le \frac{f(t)}{\Pr(t-d'<S<t-d)} \le \frac{\lambda'}{p} .\]
Thus, \eqref{eq:1-hit-regularity} will follow if we show that
\begin{equation}\label{eq:1-hit-bounded}
\frac{\Pr(t-d'<S<t-d)}{\Pr(S+S' > t)} \qquad\text{is bounded above and below as $t \to \infty$}.
\end{equation}
Since $p<1-\cL(T|_{A^c})^{-1}$, \cref{lem:exp-geom} implies that there exist $c>0$ and $0<\nu<1$ such that
\[ \Pr(S > t) = c \nu^t (1+o(1)) \qquad\text{as }t \to \infty .\]
It easily follows from this that $\Pr(t-d'<S<t-d)=\Theta(\nu^t)$ and $\Pr(S+S'>t)=\Theta(\nu^t)$ as $t\to \infty$.
Therefore, \eqref{eq:1-hit-bounded} holds and the proof is complete.
\end{proof}

\section{Proof of Theorem~\ref{thm:RP2}}
\label{sec:proof-main-thm}

We now explain how the `if' part of \cref{thm:RP2} follows from the results given in \cref{sec:renewal-with-regular-jd} and \cref{sec:regularize-jd}. The proof of the `only if' part is given in a separate proposition below.

\begin{proof}[Proof of `if' part of \cref{thm:RP2}]
Let $X$ be a stationary renewal point process whose jump distribution $T$ has exponential tails and a non-singular convolution power.
Recall that non-singular jump distributions are continuously-divisible and that, by \cref{lem:regularize-T-exp}, $T^*_{k\text{-hit}(A)}$ has exponential tails for any $A$ and $k$.
Thus, \cref{lem:stopping-time} and \cref{lem:regularize-T-power} imply that by replacing $T$ with $T^*_{k\text{-hit}(A)}$ for a suitable $A$ and $k$, we may assume that $T$ is non-singular.
\cref{lem:stopping-time} and \cref{lem:regularize-T-nonsing} now imply that by replacing $T$ with $T^*_{2\text{-hit}(A)}$ for some other suitable $A$, we may assume that $T$ is absolutely continuous with a continuous density. \cref{lem:stopping-time} and \cref{lem:regularize-T-abscont} now imply that by replacing $T$ with $T^*_{1\text{-hit}(A)}$ for yet another $A$, we may assume that $T$ is absolutely continuous with a density $f$ satisfying~\eqref{eq:regular-density}. Finally, \cref{prop:RP-special-case} yields that $X$ is finitarily isomorphic to a Poisson point process.
\end{proof}

The following proposition shows that the assumptions of \cref{thm:RP2} are necessary even if one only desires to obtain a finitary factor of a Poisson point process. In particular, it implies the `only if' part of \cref{thm:RP2}.

\begin{prop}\label{prop:necessary}
If a stationary renewal point process is a finitary factor of a Poisson point process, then its jump distribution has exponential tails and a non-singular convolution power.
\end{prop}

Throughout the proof, we use two simple observations. First, given a real-valued random variable $T$ and a positive-probability event $A$, if $T$ is singular then so is $T|_A$ (where $T|_A$ denotes a random variable having the conditional distribution of $T$ given that $T \in A$), and if $T$ is absolutely continuous then so is $T|_A$. This is easily seen by writing the law of $T$ as a mixture of the laws of $T|_A$ and $T|_{A^c}$. Second, if $T_1$ and $T_2$ are independent, then $T_1+T_2$ is non-singular whenever $T_1$ is, and it is absolutely continuous whenever $T_1$ is.

\begin{proof}
Let $X$ be a stationary renewal point process with jump distribution $T$.
Let $0<p_1<p_2<\cdots$ be the points of $X$ in $(0,\infty)$, and define $T_n := p_{n+1}-p_n$ for $n \ge 1$. By the definition of a renewal point process, $\{T_n\}_{n=1}^\infty$ are independent copies of $T$.
For $t \ge 0$, let $N_t$ be the number of points of $X$ in $(0,t]$ and let $S_t$ be the smallest point of $X$ in $(t,\infty)$. Note that $S_t-S_0 = T_1+\cdots+T_{N_t}$.

Suppose now that $X$ is a finitary factor of a Poisson point process $\Lambda$.
Note that $S_0<1$ if and only if $X$ has a point in $(0,1)$. Since this event has positive probability and since the coding window $W$ is almost surely finite, there exists a finite $r$ such that $\Pr(S_0<1,\,W<r)>0$. Thus, there exists a positive-probability event $F$ on which $S_0<1$ and such that $S_0 \cdot \1_F$ is determined by $\Lambda \cap (-r,r)$ (since our definition of finitary does not require the coding window to be a stopping time, $F$ might not be the event $\{S_0<1,W<r\}$ itself). Define the translated events $F_n := \theta_{-2nr}(F)$ and note that $\{F_n\}_{n \in \Z}$ are independent.

Since $S_1<1+2nr$ on the event $F_n$ for $n \ge 0$, it follows from the independence the events $\{F_n\}_n$ that $S_1$ has exponential tails. Since $S_1 \ge T_1$ on the event that $N_1 \ge 1$ and since $\Pr(N_1 \ge 1)>0$, it easily follows that $T$ has exponential tails.

Let us now show that $T$ has a non-singular convolution power. The law of $S_t-S_0$ is a mixture of the conditional laws of $T_1+\cdots+T_n$ given that $N_t=n$. If all convolution powers of $T$ are singular, then so are the latter conditional laws, and hence so is $S_t-S_0$. Thus, it suffices to show that $S_t-S_0$ is non-singular for some $t$. We shall show this for $t=2r$.
To see this, first observe that $S_0$ and $S_{2r}$ are conditionally independent given $F \cap F_1$. Thus, to show that $S_{2r}-S_0$ is non-singular, it suffices to show that the conditional law of $S_0$ given $F$ is non-singular.
In fact, we claim that this conditional law is absolutely continuous. To this end, it suffices to show that the unconditional law of $S_0$ is absolutely continuous. This, in turn, can be seen using the Palm distribution inversion formula~\cite[Proposition~11.3]{kallenberg2006foundations}, which says that $\E[f(X)] = \frac{1}{\E \mathsf S_0} \cdot \E \int_0^{\mathsf S_0} f(\theta_s(Y))ds$ for any non-negative measurable function $f$ and for some (non-stationary) point process $Y$ called the Palm distribution of $X$, where $\mathsf S_0$ denotes the smallest point of $Y$ in $(0,\infty)$. Indeed, applying this formula with $f(X):=\1_{\{S_0>t\}}$ yields that $\Pr(S_0>t)=\frac{\E \max\{\mathsf S_0 - t,0\}}{\E \mathsf S_0}$, which shows that $S_0$ has density $g(t) := \frac{\Pr(\mathsf S_0>t)}{\E \mathsf S_0}$.
\end{proof}

\medskip
\paragraph{\textbf{Acknowledgments}}
I am grateful to Nishant Chandgotia, Edwin Perkins and Zemer Kosloff for helpful discussions and comments.

\bibliographystyle{amsplain}
\bibliography{library}

\end{document}